\documentclass[12pt]{amsart} 

\usepackage{graphicx,amssymb,colordvi,textcomp,latexsym}

\usepackage{tikz}

\usepackage{caption}
\usepackage{subcaption}
\usepackage[title]{appendix}
\usepackage[showframe=false]{geometry}
\usepackage{changepage}

\usepackage{ulem} %para usar \sout{}

\usepackage{tikz}
\usetikzlibrary{arrows,shapes, positioning, matrix, patterns, decorations.pathmorphing, arrows.meta,automata}

\usepackage{color}
\usepackage{colortbl}
\usepackage{soul}

\usepackage{array}
\usepackage[hidelinks]{hyperref}
\usepackage{amsfonts}
\usepackage{amsmath}
\usepackage{amsthm}
\usepackage{adjustbox}
\usetikzlibrary{arrows,shapes, positioning, matrix, patterns, decorations.pathmorphing}

\newcommand{\R}{\mbox{$\RR$}}

\renewcommand{\Im}{\operatorname{Im}}

\newcommand{\RR}{\mbox{$\mathbf R$}}

\newtheorem{lemma}{Lemma}[section]

\newtheorem{prop}[lemma]{Proposition}

\newtheorem{cor}[lemma]{Corollary}

\theoremstyle{definition}
\newtheorem{Defx}[lemma]{Definition}
\newenvironment{Def}
  {\pushQED{\qed}\Defx}
  {\popQED}
\newtheorem{examx}[lemma]{Example}
\newenvironment{exam}
  {\pushQED{\qed}\examx}
  {\popQED}

\theoremstyle{remark}
\newtheorem{remx}[lemma]{Remark}
\newenvironment{rem}
  {\pushQED{\qed}\remx}
  {\popQED}
\newtheorem{rems}[lemma]{Remarks}

\title[]{The Bifurcation Coalescence Problem for Feedforward Coalescence Networks}
\author{Manuela Aguiar}
\address{Manuela Aguiar, Faculdade de Economia, Centro de Matem\'atica, Universidade do Porto,
Rua Dr Roberto Frias, 4200-464 Porto, Portugal.}
\email{maguiar@fep.up.pt}
\author{Pedro Soares}
\address{Pedro Soares,  Instituto Superior de Economia e Gest\~ao, Universidade de Lisboa,
Rua do Quelhas, n$^{\underline{o}}$ 6, 1200-781 Lisboa, Portugal.}
\email{psoares@iseg.ulisboa.pt}

\date{\today}

\begin{document}

\begin{abstract}

Consider two networks and combine them through the coalescence operation to get a larger network. 
Is it possible to study the steady-state bifurcations of the coalescence network by studying the steady-state bifurcations of the component networks? 
We conclude that this is not possible for general coalescence networks. 
We show, however, that in the case of feedforward coalescence networks this is possible and we cover the simplest cases.
In particular, we prove how the growth rate of the bifurcation branches in the feedforward coalescence network depends on the connections from one network to the other. 
\end{abstract}

\maketitle

\tableofcontents

\section{Motivation}

In general, real-world networks are large, which makes the study of their dynamical properties a difficult task.
An interesting methodology would be to be able to divide a network into smaller subnetworks, study their dynamical properties separately and then draw conclusions about the dynamics of the large network. This, in analogy, for example, with what is done in the analysis of electronic circuits.

Network motifs are small networks that appear recurrently as components of large networks and that can be considered as their building-block components.  Since it is easier to analyze motif dynamics, it is useful to be able to relate the dynamical properties of a large network with those of its motif components. 

In this work, among the binary operations that can be used to combine small networks in order to get larger ones, we choose to consider the coalescence operation.
%Here we consider networks that can be seen as the coalescence of subnetworks and w
Roughtly speaking, a network $N$ is the coalescence of two networks $N_1$ and $N_2$, which we denote by $N_1 \circ N_2$, if $N$ can be obtained by identifying (merging) one cell of $N_1$ with one cell of $N_2$, where the {\it coalescence cell} inherits the input and output edges of those two {\it merging cells}. We consider networks where the cells are all of the same type, there is only one type of edge and the edges can have associated weights.

We are interested in analyzing what conclusions can be drawn for steady-state bifurcations of the coalescence networks by studying the corresponding steady-state bifurcations in the component networks. 
Given a codimension-one steady-state bifurcation problem in a coalescence network $N=N_1 \circ N_2$ we ask whether the bifurcation condition for $N$ is necessarily a bifurcation condition for at least one of the networks $N_1$ or $N_2$. 
If the answer is positive, the next question we ask is to what extent it is possible to predict the bifurcation branches for $N$ having knowledge of the bifurcation branches for $N_1$ and/or $N_2$. 
We call this the {\it bifurcation coalescence problem}.

We conclude that for most coalescence networks the bifurcation coalescence problem has a negative answer.
The first question fails, the bifurcation conditions in an $N_1\circ N_2$-admissible system does not need to follow from a bifurcation condition in $N_1$ or $N_2$-admissible systems.
Even when a coalescent network shares a bifurcation condition with the component networks, we observe that, except for some specific cases, the bifurcation branches in the coalescence network cannot, in general, be obtained from the bifurcation branches in the component networks.
However, the bifurcation coalescence problem has some positive answers if the coalescence network has a feedforward structure, that is, is a {\it  Feedforward Coalescence Network (FFCN)} -- a network $N_1 \circ N_2$ such that 
the coalescence cell in $N_2$ has no inputs from the other cells in $N_2$ and, thus, $ N_1 \circ N_2$ has a feedforward structure where the first layer contains the cells in $N_1$ and the second layer the cells in $N_2$ excluding the coalescence cell.

Restricting the bifurcation coalescence problem to FFCNs and considering diffusive admissible systems, we first show that 
 the  bifurcation conditions for a FFCN $N=N_1 \circ N_2$  are the union of the bifurcation conditions for the networks $N_1$ and $N_2$. 
We then study three different cases:\\
(1) If the bifurcation occurs exactly in one of the component networks $N_1$ or $N_2$, then each bifurcation branch in the coalescence network is given by an unique extension of a bifurcation branch in the component network where the bifurcation occurs, Proposition~\ref{Prop:bif_cond_one}. \\
(2) When the bifurcation occurs in both component networks, we consider an explicit condition on the couplings between the merging cell and the other cells in $N_2$ for nilpotent bifurcations.
This {\it coupling condition} can be formulated as the couplings from the merging cell to the other cells in $N_2$ being a linear combination of the couplings between the other cells in $N_2$.
 When the bifurcation condition is met simultaneously in both networks, and it is simple in $N_2$,
if the above coupling condition does not hold, then each bifurcation branch in $N_1$ gives rises to two bifurcation branches in the coalescence network (Proposition~\ref{prop:extbifbra}).
Moreover, these bifurcation branches have a growth rate equal to $\sqrt{|\lambda|}$. 
This coincides with previous study on the bifurcations in feedforward networks where the amplified growth-rate bifurcation branches are studied, 
Soares~\cite{S18}.\\
(3) Finally, if the above coupling condition does hold, then each bifurcation branch in $N_1$ still extends to two bifurcation branches in the coalescence network (Proposition~\ref{prop:extbif2}). 
However, in this case, the bifurcation branches do not show the amplified growth-rate.
Thus the amplified growth-rate can be explained by the relation between the couplings between the merging cell with the other cells in $N_2$ and the couplings between the other cells in $N_2$.

The results mentioned above include non-degenerated conditions, as usual in bifurcation theory. 
These non-degenerated conditions explicitly depend on the network structure and the bifurcation branch that we are extending.
Fixing a network, the non-degenerated condition is satisfied for generic admissible system of that network or they fail for any admissible system of that network.
For each of our results, we present examples of networks where the non-degenerated conditions generically hold and the previous results can be applied.
Moreover, we present examples of networks such that the non-degenerated conditions do not hold and we illustrate how the methods introduced here can be applied in those cases.

We revisit a network analyzed in \cite{ADGL09}, where the authors used a different approach considering quotient networks instead of subnetworks. 
More concretely, given a coupled cell network there are subspaces, defined by equalities of cell coordinates, that are left invariant by all the coupled systems that are admissible by that network and that are designated by  the  {\it  synchrony subspaces} of the network. 
The restriction of an admissible system to a synchrony subspace is a coupled system that is admissible by a smaller network, a  {\it  quotient network}. 
In \cite{ADGL09}, the question is, given a codimension-one steady-state bifurcation problem in a network, if it is possible to get all the bifurcation branches for that network by studying the corresponding bifurcation problem in the restriction to its quotient networks where the bifurcation condition is valid. 
The conclusion is that the answer is negative, in general. 
One example, presented in \cite{ADGL09}, where the answer is negative is the network represented in Figure~\ref{Fig:coal_3-cell_net6}.
This network is a particular example of a {\it  Feedfoward Coalescence Network (FFCN)}.
Despite this network does not satisfy the non-degenerated condition, we illustrate how the methods introduced here can be used to study the bifurcation branches of that network.

The paper is organized as follows: Section 2 contains the definitions and results on coupled cell networks and coupled admissible systems that are needed here.
In particular, it is presented the relation between the linearization of the admissible systems at the origin and the (weighted) adjacency matrix and Laplacian matrix of the network. 
In Section 3 we consider FFCNs, first we prove that the eigenvalues of the Laplacian matrix $L_{N}$, with $N=N_1 \circ N_2$ a FFCN, are given by the union of the eigenvalues of the Laplacian matrices $L_{N_1}$ and $L_{N_2}$ of the component networks $N_1$ and $N_2$. 
Then, we study the bifurcation coalescence problem for FFCNs by considering admissible systems that are diffusive.
In Section 4, we end with some final remarks.

\section{Definitions and preliminary results}

We consider the coupled cell networks formalism of Golubitsky, Stewart and collaborators~\cite{SGP03,GST05,GS06} and its extension to weighted networks in Aguiar and Dias~\cite{AD18}.
Here, we focus these formalisms to the case where the networks have only one cell type and one edge type, uniform networks. 
 A (weighted) coupled cell network can be represented by a (weighted) directed graph, where the cells are placed at vertices (nodes) and the couplings are depicted by directed arrows (edges).

% and the cell and edge equivalence relations are indicated, respectively, by different types of vertices and different types of edges in the graph. 

\begin{Def} \normalfont
A {\it (uniform coupled cell) network} $N$ consists of a finite nonempty set $C$ of {\it nodes} or {\it cells} and a finite nonempty subset $E= \{ (i,j):\ i,j \in C\}\subseteq C\times C$ of {\it edges} or {\it arrows}, where each pair $(i,j)$ represents an edge from cell $i$ to cell $j$.
% Moreover, it consists of two equivalence relations: $\sim_{C}$ on $C$ and $\sim_{E}$ on $E$ such that the {\it consistency condition} is satisfied: if $(i_1,j_1) \sim_{E} (i_2,j_2)$, then $i_1 \sim_{C} i_2$ and $j_1 \sim_{C} j_2$.
And we write $N = (C,E)$.
In a {\it weighted (uniform coupled cell) network}, each edge $(i,j)$ has associated a weight $w_{ij} \in \RR$.
%\hfill $\Diamond$
\end{Def} 

A network is said {\it connected} if the associated graph is connected -- there is an undirected path between any two vertices (cells). 
From now on, when we refer to a network we mean a connected network.

%We consider weighted networks with only one cell type and one edge type -- {\it weighted uniform networks}.
The topology of an $n$-cell weighted network $N$ can be given by an  $n \times n$ {\it (weighted) adjacency matrix}, $W_N = [w_{ij}]_{1 \le i,j \le n}$, where $w_{ij}$ is the weight of the edge from cell $j$ to cell $i$ or zero if there is no such edge. 
Thus, we can simplify notation and use simply $N = (C,E, W_N)$. 

%Given cells $c$ and $d$ of $\NN$, we say that $c$ and $d$ are connected if there exists a sequence of cells $c_0, c_1, . . . , c_{l-1}, c_l$ such that $c_0 = c, c_l = d$ and there is an edge from $c_{j-1}$ to $c_j$ or an edge from $c_j$ to $c_{j-1}$, for every $1 \leq j \leq l$.  
%In this work, we always consider connected networks, i.e., networks where every two distinct cells are connected. 

The {\it input valency} of a cell $i \in C$, denoted by $v(i)$, is the sum of the weights of the edges directed to the cell $i$, that is, $v(i) = \sum_{j \in C} w_{ij}$. A network is said to be {\it regular} when all the network cells have the same input valency, that is, $v(i) = v(j)$, for all $i,j \in C$. 
This equivalent to the weighted matrix $W_N$ having constant row sum, say $v_W = \sum_{k=1}^n w_{ik}$, for $i= 1, \ldots, n$. 
In that case, we also say that $W_N$ is {\it regular} of {\it valency} $v_W$. 

Let $D_N$ be the $n \times n$ diagonal matrix with the principal diagonal entries $d_{ii}$ given by the input valencies $v(i)$. The {\it (weighted) Laplacian matrix} of $N$ is the matrix $L_N=D_N-W_N$. That is, $L_N = [l_{ij}]_{n\times n}$, with: 
$$l_{ij} = 
\left\{ 
\begin{array}{l}
-w_{ij} \mbox{  if }  i\not=j; \\
v(i) - w_{ii} \mbox{  if }  i=j\, .
\end{array}
\right.
$$

Following \cite{AR12}, we define {\it coalescence} on coupled cell networks:

\begin{Def}\label{def:coalescence}
A network $N = (C,E, W_N)$ is a {\it coalescence} of the networks $N_1 = (C_1,E_1, W_{N_1})$ and $N_2 = (C_2,E_2, W_{N_2})$, denoted by $N= N_1 \circ N_2$,  if
\begin{itemize}
\item[] $C =  C_1 \cup C_2$ and $\left(C_1 \cap C_2\right) = \{c\}$; 
\item[] $E = E_1 \cup E_2$.
\end{itemize}
and the (weighted) adjacency matrix $W_N=(w^N_{ab})_{a,b\in C}$ is the following:
\begin{itemize} 
\item[] $w^N_{cc}=w^{N_1}_{cc}+w^{N_2}_{cc}$,
\item[] $w^N_{ab}=w^{N_i}_{ab}$ if $a,b\in C_i$, $i=1,2$ and $(a,b)\neq (c,c)$.\qedhere %\hfill $\Diamond$
\end{itemize}
\end{Def}

Given two disjoint networks $N_1$ and $N_2$, we can first rename the cells $c_1\in N_1$ and $c_2\in N_2$ to ``$c$'', then apply the coalescence of Definition \ref{def:coalescence}. 
This will corresponds to the coalesced graph obtained by merging $c_1$ and $c_2$.
The coalescence of two networks depends on the selected merging cells.
In order to make it clear, we say that $N$ is the coalescence network obtained by merging cells $c_1$ and $c_2$.
We will use the term {\it coalescence cell} to refer to both $c_1$, $c_2$ and $c$.
Here, we will always assume that the cells in $C_1$ are numbered such that the coalescence cell $c_1$ is the last one and that the cells in $C_2$ are numbered such that the coalescence cell $c_2$ is the first one.

\begin{exam}
As an example, consider Figure~\ref{Fig:coal_2-cell} where the coalescence network $N = N_1 \circ N_2$ is obtained by merging cells $2$ and $3$ of the two (general) 2-cell weighted network motifs $N_1$ and $N_2$.
 
\begin{figure}[!h]
\begin{center}
\begin{tabular}{ccc}
\begin{tikzpicture}
 [scale=.25,auto=left, node distance=1.5cm]
 \node[style={circle,draw}] (n1) at (4,0) {\small{1}};
  \node[style={circle,draw}] (n2) at (14,0) {\small{2}};
  \path
(n1) edge [loop above] node {{\tiny $w_{11}$}} (n1)
(n2) edge [loop above] node {{\tiny $w_{22}$}} (n2)
 (n1) edge[->, thick, bend left=30]  node  [above=0.1pt] {{\tiny $w_{21}$}} (n2)
 (n2) edge[->, thick, bend left=30,]  node  [below=0.1pt] {{\tiny $w_{12}$}} (n1); 
\end{tikzpicture}  \qquad & \qquad 
\begin{tikzpicture}
 [scale=.25,auto=left, node distance=1.5cm]
 \node[style={circle,draw}] (n3) at (4,0) {\small{3}};
  \node[style={circle,draw}] (n4) at (14,0) {\small{4}};
  \path
(n3) edge [loop above] node {{\tiny $w_{33}$}} (n3)
(n4) edge [loop above] node {{\tiny $w_{44}$}} (n4)
 (n3) edge[->, thick, bend left=30]  node  [above=0.1pt] {{\tiny $w_{43}$}} (n4)
 (n4) edge[->, thick, bend left=30,]  node  [below=0.1pt] {{\tiny $w_{34}$}} (n3); 
\end{tikzpicture}  \qquad & \qquad 
\begin{tikzpicture}
 [scale=.25,auto=left, node distance=1.5cm]
 \node[style={circle,draw}] (1) at (4,0) {\small{1}};
  \node[style={circle,draw}] (23) at (14,0) {c};
  \node[style={circle,draw}] (4) at (24,0) {\small{4}};
 \path
(1) edge [loop above] node {{\tiny $w_{11}$}} (1)
(23) edge [loop above] node {{\tiny $w_{22} + w_{33}$}} (23)
 (1) edge[->, thick, bend left=30]  node  [above=0.1pt] {{\tiny $w_{21}$}} (23)
 (23) edge[->, thick, bend left=30,]  node  [below=0.1pt] {{\tiny $w_{12}$}} (1)
(4) edge [loop above] node {{\tiny $w_{44}$}} (4)
 (23) edge[->, thick, bend left=30]  node  [above=0.1pt] {{\tiny $w_{43}$}} (4)
 (4) edge[->, thick, bend left=30,]  node  [below=0.1pt] {{\tiny $w_{34}$}} (23); 
 \end{tikzpicture} 
\end{tabular}
\end{center}
\caption{Two 2-cell weighted network motifs $N_1$ and $N_2$ (left and center) and the coalescence network $N = N_1 \circ N_2$ (right), obtained by merging cells $2$ and $3$.} \label{Fig:coal_2-cell}
\end{figure}
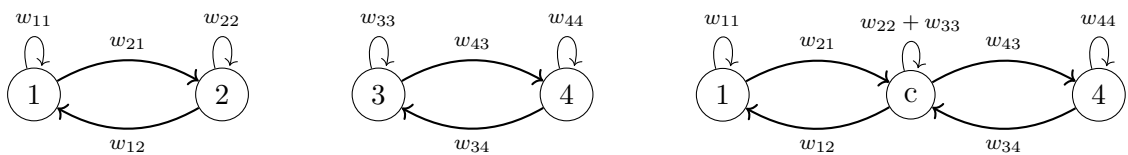

The weighted adjacency matrices of these three networks are:
$$
W_{N_1} =
\left[
\begin{array}{cc}
w_{11} & w_{12} \\
w_{21} & w_{22} \\
\end{array}
\right], \quad
W_{N_2} =
\left[
\begin{array}{cc}
w_{33} & w_{34} \\
w_{43} & w_{44} \\
\end{array}
\right], \quad \mbox{and} \quad
W_{N} =
\left[
\begin{array}{ccc}
w_{11} & w_{12} & 0 \\
w_{21} & w_{22} + w_{33}  & w_{34} \\
0 & w_{43} & w_{44} \\
\end{array}
\right]\, .
$$
\end{exam}

In general, up to numbering of the cells, the weighted adjacency matrix, the diagonal valency matrix and the weighted Laplacian matrix of the coalescence network have the following block structure 
\begin{equation} \label{Eq:block_form}
\begin{bmatrix}
   \underline{\overline{*}}_{N_1} & {*_{N_1}}^c         & 0\\
    {*_{N_1}}_c   & {*_{N_1}}^c_c +{*_{N_2}}^c_c & {*_{N_2}}_c\\
    0     & {*_{N_2}}^c & \underline{\overline{*}}_{N_2}\\
  \end{bmatrix},
\end{equation}
where $*$ is replaced by $W$, $D$ or $L$ for the weighted adjacency matrix, the diagonal valency matrix or the weighted Laplacian matrix, respectively.
Moreover,
$\underline{\overline{*}}_{N_1}$ is the $(n_1-1)\times (n_1-1)$ matrix, ${*_{N_1}}^c$ is the $(n_1-1)$-column matrix, ${*_{N_1}}_c$ is the $(n_1-1)$-row matrix, ${*_{N_1}}^c_c,{*_{N_2}}^c_c\in\RR$, $\underline{\overline{*}}_{N_2}$ is the $(n_2-1)\times (n_2-1)$ matrix, ${*_{N_2}}^c$ is the $(n_2-1)$-column matrix, ${*_{N_2}}_c$ is the $(n_2-1)$-row matrix such that the  corresponding  matrix of $N_1$ and $N_2$ have, respectively, the form
$$
\begin{bmatrix}
    \underline{\overline{*}}_{N_1} & {*_{N_1}}^c\\
    {*_{N_1}}_c & {*_{N_1}}_c^c 
  \end{bmatrix},
	\quad \quad 
	\begin{bmatrix}
    {*_{N_2}}_{c}^c & {*_{N_2}}_c\\
    {*_{N_2}}^c & \underline{\overline{*}}_{N_2} 
  \end{bmatrix}.
$$

\begin{rem}
Note that a coalescence of two regular networks will not be regular. 
If $N_1$ is regular of valency $v_1$ and $N_2$ is regular of valency $v_2$, then the merging cell has valency $v_1+v_2$ while the other cells have valency $v_1$ or $v_2$.
Nevertheless, $N_1 \circ N_2$ can be regular without $N_1$ and/or $N_2$ being regular. 
For example, if $N_1$ is regular of valency $v$ and all the cells in $N_2$ have valency $v$, with the exception of a cell $c$ that has no input, then the coalescence merging any cell of $N_1$ and the cell $c$ of $N_2$ results in a regular network.
%\hfill $\Diamond$ 
\end{rem}

Later, we will study networks with a feedforward structure defined as follows.

\begin{Def}
We say that a network has a {\it feedforward structure} if its set of cells can be partioned into an ordered sequence of subsets (layers) such that the cells in each subset can only receive edges from cells of its own subset or of the previous one.
A coalescence network $N = N_1 \circ N_2$ where the coalescence cell has no input edge in $N_2$ has a feedforward structure and will be designated by a {\it Feedforward Coalescence Network (FFCN)}.
\end{Def}

To each coupled cell network we can associate a class of dynamical systems. 
Here, we consider differential equations and assume that they have an {additive input structure}, which naturally accommodates the weighted setup, see ~\cite{AD18}. 
Since we are considering only one cell type, the internal cell dynamics is the same for all cells and, since we are assuming only one edge type, the coupling function is the same for all couplings. 
As we are going to consider codimension-one bifurcation problems, we include a bifurcation parameter $\lambda \in \RR$.
An \textit{input additive coupled cell system admissible} by the network $N$ is, then, defined by 
\begin{equation} \label{eq:systemglobal}
\dot X = F(X, \lambda),
\end{equation}
 with $F=(f_1, \ldots, f_n)$, $X=(x_1,\ldots, x_n) \in \left(\RR^k\right)^n$ and
\begin{equation} \label{eq:parEDOsystem}
\dot{x}_i  =  f_i(x_1,\ldots, x_n,\lambda) = g(x_i, \lambda) +  \sum_{j} w_{ij} h(x_i,x_j, \lambda),
\end{equation}
where the smooth functions $g:\RR^k \times \RR \rightarrow \RR^k$ and $h:\RR^k \times \RR^k \times \RR \rightarrow \RR^k$ define the internal and the coupling dynamics of the cells, respectively.

Here, we will consider that the individual cell dynamics is one-dimensional, that is, $k=1$.

\begin{rem}\label{rem:admcolsys} 
Given smooth functions $g:\RR \times \RR \rightarrow \RR$ and $h:\RR \times \RR \times \RR \rightarrow \RR$, denote by $F^N:\RR^{n_1+n_2}\times\RR\rightarrow \RR^{n_1+n_2}$, $F^{N_1}:\RR^{n_1}\times\RR\rightarrow \RR^{n_1}$ and $F^{N_2}:\RR^{n_2}\times\RR\rightarrow \RR^{n_2}$ the admissible systems for $N$, $N_1$ and $N_2$, respectively. 
An admissible system for the coalescence network $N=N_1\circ N_2$ by merging cell $c$ can be written in terms of the correspondent admissible systems for the networks $N_1$ and $N_2$ as follows, for $i=1,2$,
$$(F^N)_a(x_1,x_c,x_2,\lambda)= (F^{N_i})_a(x_i,x_c,\lambda),\quad a\in N_i\setminus \{c\}$$
$$(F^N)_c(x_1,x_c,x_2,\lambda)= (F^{N_1})_c(x_1,x_c,\lambda)+(F^{N_2})_c(x_2,x_c,\lambda)-g(x_c,\lambda),$$
where $x_i\in\RR^{n_i-1}$ corresponds to the coordinates of the cells in $N_i\setminus \{c\}$, $x_c\in\RR$ corresponds to the {\it coalescence coordinate} -- the coordinate of the coalescence cell $c$ -- and $\lambda\in\RR$.
\end{rem}

\begin{exam}
Considering, again, the networks in Figure~\ref{Fig:coal_2-cell}, the coupled cell systems with additive input structure consistent with $N_1$ and $N_2$ and the coalescence network $N$,  have the form:
$$
\left\{ 
\begin{array}{rcl}
\dot{x}_1 & = & g(x_1, \lambda) +  w_{11} h(x_1,x_1, \lambda)  +  w_{12} h(x_1,x_2, \lambda) \\
\dot{x}_2 & = & g(x_2, \lambda) + w_{21} h(x_2,x_1, \lambda) +  w_{22} h(x_2,x_2, \lambda)\\
\end{array} \, ,
\right.
$$
$$
\left\{ 
\begin{array}{rcl}
\dot{x}_3 & = & g(x_3, \lambda) +  w_{33} h(x_3,x_3, \lambda)  +  w_{34} h(x_3,x_4, \lambda) \\
\dot{x}_4 & = & g(x_4, \lambda) + w_{43} h(x_4,x_3, \lambda) +  w_{44} h(x_4,x_4, \lambda)\\
\end{array} \, ,
\right.
$$
and
$$
\left\{ 
\begin{array}{rcl}
\dot{x}_1 & = & g(x_1, \lambda) +  w_{11} h(x_1,x_1, \lambda)  +  w_{12} h(x_1,x_{23}, \lambda) \\
\dot{x}_{23} & = & g(x_{23}, \lambda) + w_{21} h(x_{23},x_1, \lambda) +  \left(w_{22} + w_{33} \right) h(x_{23},x_{23}, \lambda) + w_{34} h(x_{23},x_4, \lambda)\\
\dot{x}_4 & = & g(x_4, \lambda) + w_{43} h(x_4,x_{23}, \lambda) +  w_{44} h(x_4,x_4, \lambda)\\
\end{array} \, ,
\right.
$$
respectively, where $g$ and $h$ are smooth functions and $\lambda \in \RR$ is the bifurcation parameter.
\end{exam}

We assume that system (\ref{eq:systemglobal}) has a full synchronous equilibrium when $\lambda=0$, which we can consider is at the origin.
The next proposition follows from Proposition 3.1 in \cite{AD18}, which states that the linearization of system (\ref{eq:systemglobal}) at the origin is determined by the adjacency matrix of the associated network.

\begin{prop} \label{prop:Jac}
Let $N$ be an $n$-cell weighted coupled cell network with weighted adjacency matrix $W_N$ and weighted Laplacian $L_N$.
% and denote by $\mu_1, \ldots, \mu_n$ the eigenvalues of $A_G$. 
Consider an $N$-admissible system where  the equation of cell  $i$  is given by (\ref{eq:parEDOsystem}). Take 
$$
 g_0 = \left(d_{x} g \right)_{0}, \quad h_1 =  \left(d_{x_1} h \right)_{0},  \quad h_2 =  \left(d_{x_2} h \right)_{0},
 $$ 
 the first derivatives of $g(x,\lambda)$ and $h(x_1,x_2,\lambda)$ at the origin $x=x_1=x_2=\lambda=0$. Then  
$$
J^N \equiv \left(dF\right)_{0} = g_0  {\mathrm id}_n +  {h_1} D_N + {h_2} W_N
$$
where $ {\mathrm id}_n$ is the $n \times n$ identity matrix.

\vspace{0.05in}

Moreover,

\noindent (i) If $N$ is regular with valency $v$, the Jacobian at the origin simplifies to
$$
J^N \equiv \left(dF\right)_{0} = (g_0 + v {h_1}) {\mathrm id}_n + {h_2} W_N\, .
$$
Denoting by $\mu_1, \ldots, \mu_n$ the eigenvalues of $W_N$,  the eigenvalues of $J$ are $g_0 + v {h_1} + \mu_j {h_2}$, for $j=1, \ldots, n$, including algebraic multiplicity. 
%Moreover, if $u \in {\bf C}^k$ is an eigenvector of $g_0 + v {h_1}$ and $w \in {\bf C}^n$ is an eigenvector of $W_G$ then $u \otimes w$ is an eigenvector of $J$.

\vspace{0.05in}

\noindent (ii) If ${h_2} = - {h_1}$, the Jacobian at the origin simplifies to
$$
J^N \equiv \left(dF\right)_{0} = g_0 {\mathrm id}_n +  {h_1} L_N\, . 
$$
Denoting by $\mu_1, \ldots, \mu_n$ the eigenvalues of $L_N$, the eigenvalues of $J$ are given by  $g_0  + \mu_i {h_1}$, for $i=1, \ldots, n$, including algebraic multiplicity. 
%Moreover, if $u \in {\bf C}^k$ is an eigenvector of $g_0$ and $w \in {\bf C}^n$ is an eigenvector of $L_G$ then $u \otimes w$ is an eigenvector of $J$.  
\end{prop}

A bifurcation condition associated with $N$ corresponds to a vanishing eigenvalue of $J^N$.
Moreover, we assume generically on the functions $g$ and $h$ that define the admissible systems that exactly one eigenvalue satisfies the bifurcation condition and the others do not vanish since we consider codimension-one bifurcations.
We say that $b:D\rightarrow \RR^{n}$ is an {\it equilibrium branch} of $F^N$ if $F^N(b(\lambda),\lambda)=0$, $b(0)=(0,\dots,0)$ and $D$ has the form $]-\epsilon,\epsilon[$, $[0,\epsilon[$ or $]-\epsilon,0]$ for some $\epsilon>0$. 
We say that $b$ is a { \it bifurcation branch} if $b\not\equiv 0$. 
In this work, we assume that the bifurcations in the component networks are well-known and intend to deduce the bifurcations in the coalescence network. 
For this purpose, the first question that arises is if there are common eigenvalues associated with the component networks and the coalescence network.

\begin{exam}
Returning to the networks in Figure~\ref{Fig:coal_2-cell} and fixing every weight to be equal to $1$, the Jacobian matrices at the origin of an input additive coupled cell system admissible for the networks $N_1$ and $N_2$ and the coalescence network $N$ are given by:
$$
J_{N_1} =J_{N_2} =
\left[
\begin{array}{cc}
g_0  +  2 {h_1}  + {h_2} & {h_2} \\
{h_2} & g_0  +  2 {h_1}  + {h_2} \\
\end{array}
\right]$$
and
$$J_{N} =
\left[
\begin{array}{ccc}
g_0  +  2 {h_1}  + {h_2} & {h_2} & 0 \\
{h_2} & g_0  +  4 {h_1}  + 2{h_2}  & {h_2} \\
0 & {h_2} & g_0  +  2 {h_1}  + {h_2} \\
\end{array}
\right]\,,
$$
with $g_0, {h_1}$ and ${h_2}$ given in Proposition~\ref{prop:Jac}.   
The eigenvalues of $J_{N_1}$ and $J_{N_2}$ are $g_0+2{h_1}$ and $g_0+2{h_1}+ 2{h_2}$. 
The eigenvalues of $J_N$ are $g_0+2{h_1}+{h_2}$, $ (2 g_0 + 6 h_1 + 3 h_2 + \sqrt{4 h_1^2 + 4 h_1 h_2 + 9 h_2^2})/2$ and $ (2 g_0 + 6 h_1 + 3 h_2 - \sqrt{4 h_1^2 + 4 h_1 h_2 + 9 h_2^2})/2$.
\end{exam}

As the previous example shows, the eigenvalues of $J_{N_1}$ and $J_{N_2}$ do not need to be eigenvalues of the coalescence Jacobian matrix $J_{N_1\circ N_2}$.
Thus, in general, a bifurcation condition in an $N_1\circ N_2$-admissible system does not need to follow from a bifurcation condition in the corresponding systems admissible by
the component networks $N_1$ or $N_2$.

\begin{rem}
Relatively to the block structures in (\ref{Eq:block_form}), the block structure of the Jacobian matrix of a coalescence network has a subtle difference in the entry of the merging cell.
The Jacobian of an $N_1\circ N_2$-admissible system has the following block structure form
\begin{equation} \label{Eq:block_formJ}
\begin{bmatrix}
   \underline{\overline{J}}_{N_1} & {J_{N_1}}^c         & 0\\
    {J_{N_1}}_c   & {J_{N_1}}^c_c +{J_{N_2}}^c_c -g_0 & {J_{N_2}}_c\\
    0     & {J_{N_2}}^c & \underline{\overline{J}}_{N_2}\\
  \end{bmatrix},
\end{equation}
where $\underline{\overline{J}}_{N_1}$ is the $(n_1-1)\times (n_1-1)$ matrix, ${J_{N_1}}^c$ is the $(n_1-1)$-column matrix, ${J_{N_1}}_c$ is the $(n_1-1)$-row matrix, ${J_{N_1}}^c_c,{J_{N_2}}^c_c\in\RR$, $\underline{\overline{J}}_{N_2}$ is the $(n_2-1)\times (n_2-1)$ matrix, ${J_{N_2}}^c$ is the $(n_2-1)$-column matrix, ${J_{N_2}}_c$ is the $(n_2-1)$-row matrix such that the Jacobians of the systems admissible by $N_1$ and $N_2$ have, respectively, the form
\[
\begin{bmatrix}
    \underline{\overline{J}}_{N_1} & {J_{N_1}}^c\\
    {J_{N_1}}_c & {J_{N_1}}_c^c 
  \end{bmatrix},
	\quad \quad 
	\begin{bmatrix}
    {J_{N_2}}_{c}^c & {J_{N_2}}_c\\
    {J_{N_2}}^c & \underline{\overline{J}}_{N_2} 
  \end{bmatrix}. \qedhere
\]
\end{rem} 

Despite, in general, there is no correspondence between the eigenvalues of the Jacobian matrix associated with the coalescence network and those of the Jacobian matrices associated with the component networks, there are some cases in which we can relate them.
The next result shows one of these cases.

\begin{prop}\label{prop:eigevaluenoninter}
Let $N=N_1 \circ N_2$ be a coalescence network obtained by merging cell $c$. We have, \\
\noindent (i) If $v = (v_1,v_c) \in \RR^{n_1-1} \times \RR$ is an eigenvector of  $J_{N_1}$  associated with the eigenvalue $\mu$ and such that $v_c=0$ then $w=(v_1,0,0) \in \RR^{n_1-1} \times \RR \times \RR^{n_2-1}$ is an eigenvector of  $J_{N}$ associated with the eigenvalue $\mu$. \\
\noindent (ii) If $v = (v_c,v_2) \in \RR \times \RR^{n_2-1}$ is an eigenvector of  $J_{N_2}$  associated with the eigenvalue $\mu$ and such that $v_c=0$ then $w=(0,0,v_2) \in \RR^{n_1-1} \times \RR \times \RR^{n_2-1}$ is an eigenvector of  $J_{N}$ associated with the eigenvalue $\mu$. 
\end{prop}

\begin{proof}
We will give the proof for $(i)$. The proof for $(ii)$ is analogous.
Consider the block structure form of $J_{N}$ in (\ref{Eq:block_formJ}) and let $v= (v_1,v_c)$ be an eigenvector of $J_{N_1}$ associated with the eigenvalue $\mu$, where $v_1$ is the vector in $\RR^{n_1-1}$ with the first $n-1$ coordinates of $v$, that is, without the $c$-coordinate.
If $v_c=0$ then $v=(v_1,0)$, $\underline{\overline{J}}_{N_1} v_1=\mu v_1$ and ${J_{N_1}}_c  v_1=0$.
Taking the vector $w=(v_1,0,0)$, we have that 
$$
\begin{bmatrix}
   \underline{\overline{J}}_{N_1} & {J_{N_1}}^c         & 0\\
    {J_{N_1}}_c   & {J_{N_1}}^c_c +{J_{N_2}}^c_c-g_0 & {J_{N_2}}_c\\
    0     & {J_{N_2}}^c & \underline{\overline{J}}_{N_2}
			\end{bmatrix}\begin{bmatrix}
   v_1\\
    0\\
    0\\
  \end{bmatrix}=
	\begin{bmatrix}
    \underline{\overline{J}}_{N_1}v_1 \\
    {J_{N_1}}_c v_1 \\
    0 \\
  \end{bmatrix}=	\begin{bmatrix}
    \mu v_1\\
    0\\
    0\\
  \end{bmatrix}= \mu w.
$$
 
Therefore $w$ is an eigenvector of $L_{N_1\circ N_2}$ associated with the eigenvalue $\mu$.
\end{proof}

%\begin{exam}
%When $w_{21} =0$ or $w_{34} =0$, in the example of Subsection~\ref{subsec:2-cellcoal}, an eigenvector $L_{N_1}$ or $L_{N_2}$ associated with  the eigenvalue $w_{12}$ or $w_{43}$ has value zero at the merged cell. 
%Thus it follows from the previous result, as in Proposition~\ref{prop:eig}, that $L_{N}$ and $L_{N_1}$ or $L_{N_2}$ share the eigenvalue $w_{12}$ or $w_{43}$.
%\end{exam}

As mentioned,
in general, an eigenvalue of $J_{N_1}$ or $J_{N_2}$ do not need to be an eigenvalue of the coalescence Jacobian matrix $J_{N_1\circ N_2}$.
Thus, a bifurcation in an $N_1\circ N_2$-admissible system does not need to occur for the corresponding $N_1$ or $N_2$-admissible system.
%\textcolor{magenta}{
%MA: A mim parece-me que a quest\~ao deveria ser colocada de forma diferente. Do g\'enero: 
%Lets us suppose that there exists a bifurcation in an $N_1\circ N_2$-admissible system at $\lambda=0$ that occurs simultaneously for the corresponding $N_1$ and $N_2$-admissible systems. 
%Can we deduce all the bifurcation branches for the network $N_1\circ N_2$ from the bifurcation branches for $N_1$ and $N_2$? 
%Is it true that `the coalescence' of a bifurcation branch for $N_1$ with a bifurcation branch for $N_2$ is a bifurcation branch for $N_1\circ N_2$? 
%Is it possible for $N_1\circ N_2$ to have other bifurcation branches besides these?}
Nevertheless, lets us suppose that there exists a bifurcation in an $N_1\circ N_2$-admissible system at $\lambda=0$ that occurs simultaneously for the corresponding $N_1$ and $N_2$-admissible systems. 
Can we deduce all the steady-state bifurcation branches for the network $N_1\circ N_2$ from the bifurcation branches for $N_1$ and $N_2$? 
Is it true that `the coalescence' of a bifurcation branch for $N_1$ with a bifurcation branch for $N_2$ is a bifurcation branch for $N_1\circ N_2$? 
Do we get all bifurcation branches for $N_1\circ N_2$ by`the coalescence' of the bifurcation branches for $N_1$ and $N_2$?
In general, the answers to these questions are negative as we detail in the following.

Given $(b_1,b^1_c)$ and $(b^2_c,b_2)$ equilibrium branches of $F^{N_1}$ and $F^{N_2}$, `the coalescence' of these branches is well defined if $b^1_c\equiv b^2_c$.
However, even under the condition $b^1_c\equiv b^2_c$, `the coalescence' of these branches $(b_1,b^1_c,b_2)$ does not need to be an equilibrium branch of $F^N$, because
$$(F^N)_d(b_1(\lambda),b_c(\lambda),b_2(\lambda),\lambda)= (F^{N_1})_d(b_1(\lambda),b_c(\lambda),\lambda)+(F^{N_2})_c(b_2(\lambda),b_c(\lambda),\lambda)-g(b_c(\lambda),\lambda)$$
$$=-g(b_c(\lambda),\lambda)$$
does not need to vanish.

If $g(b_c(\lambda),\lambda)=0$, then $b_c(\lambda)$ is an equilibrium branch of the network with only one cell and without couplings.
This network has only the eigenvalue $g_0$.
Thus the bifurcation condition in $N_1\circ N_2$, $N_2$ and $N_1$ is given by $g_0=0$ or we can assume without loss of generality that $b_c(\lambda)=0$. 
If $g_0=0$, then ${h_1} D_N v=-{h_2} W_N v$ for $v\in \ker(J^N)$ which can be a degenerated condition.
When $g_0\neq 0$, we can assume without of loss of generality that $g(x,\lambda)=0\iff x=0$ in a neighbourhood of the origin, and `the coalescence' of bifurcation branches is a bifurcation branch if and only if the component bifurcation branches vanish at the merging cell.

In the remaining of this work, we restrict to the class of admissible systems such that the condition ${h_1}=-{h_2}$ in the second part of Proposition~\ref{prop:Jac} holds and the full-synchrony subspace $\Delta_0=\{x\in\RR: x_i=x_j, \ \forall_{i,j}\}$ is invariant for every system.

\begin{rem}\label{rem:fullsysinv}
The full-synchrony subspace is invariant by an (input-additive) admissible system for a network $N$ if and only if $h(z,z,\lambda)=0$ or $N$ is regular. In fact, 
note that we have
$f_c(z,\dots,z,\lambda)=f_d(z,\dots,z,\lambda)\iff h(z,z,\lambda)(v(c)-v(d))=0$. 
The condition that the network is regular appears in \cite[Remark 2.7]{ADF17}. 
However, recall that, as mentioned earlier, in the coalescence case the networks $N_1\circ N_2$, $N_1$ and $N_2$ cannot all be regular.
Thus, here we have to consider the condition $h(z,z,\lambda)=0$ on the coupling term which is usually called {\it diffusive}.
This condition appears, for example, in \cite{NSS19,PERV14,PPP19} where the coupling function depends on the difference between the own state of the node and the state of the input node.
Moreover, it implies that the condition ${h_1} = - {h_2}$ in Proposition~\ref{prop:Jac} holds.
\end{rem}

We thus focus on the steady-state bifurcation branches for coalescence admissible systems that are diffusive.
It follows from Proposition~\ref{prop:Jac} that a bifurcation condition is given by $g_0 + \mu {h_1}=0$, where $\mu$ is an eigenvalue of $L_N$. 
In this case, we say that such {\it bifurcation condition is associated with the eigenvalue} $\mu$ of $L_N$.
Moreover, we assume generically on the functions $g$ and $h$ that define the admissible systems that exactly one eigenvalue satisfies the bifurcation condition and that $g_0 + \tilde{\mu} {h_2}\neq 0$ for the other eigenvalues $\tilde{\mu}$ of $L_N$.
We also assume that the origin is an equilibrium for any $\lambda$.

\begin{rem}
When the bifurcation condition is not given by the zero eigenvalue of $L_N$
, that is, when $g_0\neq 0$, we can assume without loss of generality that the origin is an equilibrium for any small $\lambda$, after a change of coordinates. 
This is because, by Remark~\ref{rem:fullsysinv}, the full-synchrony subspace is invariant for diffusive admissible systems and the unique eigenvalue of the Laplacian matrix restricted to the full-synchrony subspace is zero.
\end{rem}

\section{Feedforward coalescence networks (FFCNs)}

%Let $N_i$, $i=1,\ldots, r$, $r \ge 2$, be weighted networks with $n_i$ cells. In this section, we assume that $N = N_1 \circ N_2 \circ  \cdots \circ M_r$ has a feedforward structure of path components. This happens if the coalescence cell in $N_i$, $i > 2$ has no inputs from the other cells in $N_i$ (but can have self inputs) and, thus, the first row of the Laplacian matrix $L_{N_i}$ has only zero entries.

In this section, we assume that $N = N_1 \circ N_2$ is a feedforward coalescence network (FFCN). 
We first prove that, for FFCNs, the eigenvalues of the Laplacian matrix $L_{N}$ are given by the union of the eigenvalues of the Laplacian matrices $L_{N_1}$ and $L_{N_2}$.
Then we study the bifurcation coalescence problem for feedforward diffusive coalescence admissible systems.

\subsection{Eigenvalues and eigenvectors structure for FFCNs}
In what follows, given a matrix $A$, we denote by $\overline{A}$, the submatrix of $A$ obtained by removing the first row.
%; given a vector $v$, we denote by $\overline{v}$, the vector obtained from $v$ by removing the first component.

%We start by considering the case $r=2$ and then generalize for $r >2$. 

The first row of the Laplacian matrix $L_{N_2}$ has only zero entries (regardless of whether the coalescence cell has self inputs or not) and, thus, the Laplacian matrix for $N = N_1 \circ N_2$ has the following block structure %\footnote{Para a matriz Laplaciana ser relevante para o estudo das condições de bifurcação é necessário a condição de difusão ${h_1}=-{h_2}$. A estrutura por blocos é válida para FFCN se a célula de junção não recebe self inputs.}
\begin{equation} \label{eq:block_L}
L_N=
\left[
\begin{array}{cc}
L_{N_1} & | 0_{n_1 \times (n_2-1)} \\
\hline
0_{(n_2-1) \times (n_1-1)} | & \overline{L}_{N_2}
\end{array}
\right]\, .
\end{equation}

We have, then, the following result.

\begin{prop} \label{prop:eig_val}
The eigenvalues of $L_N$ are the union of the eigenvalues of $L_{N_1}$ and $L_{N_2}$. Moreover, for $\mu$ an eigenvalue of $L_N$,  \\
\noindent (i) if $\mu =0$, we have $m_a(0) = m^1_a(0) + m^2_a(0) -1$, \\
\noindent (ii) if $\mu \ne 0$, we have $m_a(\mu) = m^1_a(\mu) + m^2_a(\mu)$, \\
where $m_a(\mu)$ and $m^i_a(\mu)$, $i=1,2$, denote the algebraic multiplicity of $\mu$ for $L_N$ and $L_{N_i}$, respectively.
\end{prop}
%%%%
\begin{proof}
The result follows from the block lower triangular form of $L_N$, and the fact that $\overline{L}_{N_2}$ is the $(n_2-1) \times n_2$ submatrix of $L_{N_2}$ obtained by removing the first row, which is a zero row. 
\end{proof}

\begin{rems} \label{Rems:eigen_0}
We have, \\
\noindent (i) By Proposition 12 in \cite{CA02}, the zero eigenvalue of a Laplacian matrix is always semisimple. Thus, the geometric multiplicity of the eigenvalue zero of $L_N$ is given by $m_g(0) = m^1_g(0) + m^2_g(0) -1$. \\
\noindent (ii) We can always assume that, except for the eigenvector ${\bf 1}_{n_2}$, the other generators of the eigenspace $E_0$ for $L_{N_2}$ have the first coordinate equal to zero.
\hfill $\Diamond$
\end{rems}

%%%%%%%%%%%%%%%%%%%%%%%%%%%%

\begin{prop} \label{prop:eig_vec}
For each eigenvalue $\mu$ of $L_{N}$, the associated (generalized) eigenvectors are in correspondence with the (generalized) eigenvectors of $L_{N_1}$ and $L_{N_2}$ associated with $\mu$.
A (generalized) eigenvector of $L_{N}$ associated with $\mu$ can be written as $(v, c,w)\in\RR^{n_1-1}\times\RR\times\RR^{n_2-1}$.

\vspace{0.05in}
\noindent For $\mu  = 0$, \\
\noindent (i) $(v,c,w)$ is an eigenvector of $L_{N}$ associated with the eigenvalue $0$ iff 
$(v,c)$ and $(c,w)$ are eigenvectors of $L_{N_1}$ and $L_{N_2}$, respectively, associated with the eigenvalue $0$. 
%$(v,c)$ is an eigenvector of $L_{N_1}$ associated with the eigenvalue $0$ and $(c,w)$ is an eigenvector of $L_{N_2}$ associated with the eigenvalue $0$. 

\noindent (ii) In particular, if $(0,w)$ is an eigenvector of $L_{N_2}$ associated with the eigenvalue $0$ then $(0,0, w)$ is an eigenvector of $L_{N}$ associated with the eigenvalue $0$.

\vspace{0.05in}
\noindent For $\mu  \ne 0$,\\
\noindent (iii) If $(v,c)\neq (0,0)$, then there exists $w$ such that $(v, c,w)$ is a generalized eigenvector of $L_{N}$ associated with $\mu$ iff $(v,c)$ is a (generalized) eigenvector of $L_{N_1}$ associated with $\mu$.

\noindent (iv) If $w\neq 0$, then $(0, 0, w)$ is a (generalized) eigenvector of $L_{N}$ associated with $\mu$ iff $(0,w)$ is a (generalized) eigenvector of $L_{N_2}$ associated with $\mu$.
\end{prop}
%%%%
\begin{proof}
We start with the case $\mu=0$.
Using (\ref{eq:block_L}), we see that $$L_{N_1\circ N_2} (v,c,w)=0 \iff L_{N_1}(v,c)=0 \wedge L_{N_2}(c,w)=0.$$
So (i) follows from this equivalence. 
And (ii) follows since $L_{N_1}(0,0)=0$.

For the case $\mu \ne 0$. Let $\underline{\overline{L}}_{N_2}$, be the submatrix of $\overline{L}_{N_2}$ obtained by removing the first column. Then, we have
\begin{equation} \label{eq:block_k}
(L_N - \mu I)^k=
\left[
\begin{array}{ccc}
(L_{N_1} - \mu I)^k & |  & 0_{n_1 \times (n_2-1)} \\
\hline
A_{k}  & | & (\underline{\overline{L}}_{N_2} - \mu I)^k
\end{array}
\right]\, .
\end{equation}
Thus, if $(L_N - \mu I)^k(v,c,w)=0$, then $(L_{N_1} - \mu I)^k(v,c)=0$. 
So, if $(v,c,w)$ is a $k$-generalized eigenvector of $L_{N}$ associated with $\mu$, then $(v,c)$ is a $k$-generalized eigenvector of $L_{N_1}$ associated with $\mu$. 
If $(0,0,w)$ is a $k$-generalized eigenvector of $L_{N}$ associated with $\mu$, then $(0,w)$ is a $k$-generalized eigenvector of $L_{N_2}$ associated with $\mu$. 
This proves one side of the equivalences in (iii) and (iv).

For the other direction in (iii), let $(v,c)$ be a $k_1$-generalized eigenvector of $L_{N_1}$ associated with $\mu$ and $k\geq k_1$ be a integer such that $(L_{N_1} - \mu I)^{k}(v,c)=0$ and $\RR^{n_2-1}= \Im((\underline{\overline{L}}_{N_2} - \mu I)^{k})\oplus \ker((\underline{\overline{L}}_{N_2} - \mu I)^{k})$. Note that $(L_N - \mu I)^{k+k}(v,c,w)$ is equal to
$$\begin{bmatrix}
(L_{N_1} - \mu I)^{2k}(v,c)\\
A_k(L_{N_1} - \mu I)^{k}(v,c)+(\underline{\overline{L}}_{N_2} - \mu I)^{k} A_k(v,c)+ (\underline{\overline{L}}_{N_2} - \mu I)^{2k}w
\end{bmatrix}.$$
Moreover, $A_k(v_1, v_c)= w_I+w_k$, where $w_I\in \Im((\underline{\overline{L}}_{N_2} - \mu I)^{k})$ and $w_k\in \ker((\underline{\overline{L}}_{N_2} - \mu I)^{k})$. 
So there exists $w$ such that 
$-w_I=(\underline{\overline{L}}_{N_2} - \mu I)^{k} w$. 
Since $(L_{N_1} - \mu I)^{k}(v_1,v_c)=0$, we have that 
$$(L_N - \mu I)^{k+k}(v,c,w)= 
\begin{bmatrix}
0\\
A_k 0+(\underline{\overline{L}}_{N_2} - \mu I)^{k}(w_I+w_k)- (\underline{\overline{L}}_{N_2} - \mu I)^{k}w_I
\end{bmatrix} =(0,0).$$
Then there exists $w$ such that $(v,c,w)$ is a generalized eigenvector of $L_{N}$ associated with $\mu$ and (iii) holds. 
The other direction in (iv) follows from the first row of $L_{N_2}$ being zero and Proposition~\ref{prop:eigevaluenoninter} $(ii)$.
\end{proof}

%\begin{rem}\label{Rems:eigenvec_0}
%Recalling Remarks~\ref{Rems:eigen_0}, we know that the zero eigenvalue is always semisimple and, except for the eigenvector ${\bf 1}_{n_2}$, the other generators of the eigenspace $E_0$ for $L_{N_2}$ have the $c$-coordinate equal to zero. 
%Given a basis $(v_1,c_1), (v_2,c_2),\dots,(v_{g_1},c_{g_1})$ of the eigenspace $E_0$ for $L_{N_1}$ and a basis $(1,w_1), (0,w_2),\dots,(0,w_{g_2})$ of the eigenspace $E_0$ for $L_{N_2}$. 
%Then there exists $\hat{w}_1,\hat{w}_2,\dots,\hat{w}_{g_1}$ such that $$(v_1,c_1,\hat{w}_1), (v_2,c_2,\hat{w}_2),\dots,(v_{g_1},c_{g_1},\hat{w}_{g_1}),(0,0,w_2),\dots,(0,0,w_{g_2})$$ is a basis of the eigenspace $E_0$ for $L_N$.
%\end{rem}
\begin{cor}\label{Rems:eigenvec_0}
Given a basis $(v_1,c_1), (v_2,c_2),\dots,(v_{g_1},c_{g_1})$ of the eigenspace $E_0$ for $L_{N_1}$ and a basis $(1,w_1), (0,w_2),\dots,(0,w_{g_2})$ of the eigenspace $E_0$ for $L_{N_2}$. 
Then there exist $\hat{w}_1,\hat{w}_2,\dots,\hat{w}_{g_1} \in \RR^{n_2-1}$ such that $$(v_1,c_1,\hat{w}_1), (v_2,c_2,\hat{w}_2),\dots,(v_{g_1},c_{g_1},\hat{w}_{g_1}),(0,0,w_2),\dots,(0,0,w_{g_2})$$ is a basis of the eigenspace $E_0$ for $L_N$.
\end{cor}
\begin{proof}
Recalling Remarks~\ref{Rems:eigen_0}, we know that the zero eigenvalue is always semisimple and, except for the eigenvector ${\bf 1}_{n_2}$, the other generators of the eigenspace $E_0$ for $L_{N_2}$ can be selected to have the $c$-coordinate equal to zero. The result then follows from Proposition~\ref{prop:eig_vec}.
\end{proof}

From Proposition~\ref{prop:eig_vec}, for an eigenvalue $\mu\neq 0$, an eigenvector of $L_{N_1}$ corresponds to a generalized eigenvector of $L_N$, but not necessarily an eigenvector of $L_N$.
In the following corollary, we give a condition on the edges from the coalescence cell $c$ to the other cells of $N_2$ for a strict correspondence.
Note that the edges from cell $c$ to the other cells of $N_2$ correspond to the $c$-column of the matrix $L_{N_2}$ excluding the $c$-row. 

\begin{cor}\label{cor:semisimpleeigen}
Let $\mu\neq 0$ be an eigenvalue of $L_{N}$. And denote by $L_{N_2}^c$ the $c$-column of $L_{N_2}$ excluding the $c$-row. If $L_{N_2}^c\in \Im((\underline{\overline{L}}_{N_2} - \mu I))$, then\\
(i) $(v_1,v_c,v_2)$, for some $v_2 \in \RR^{n_2-1}$, is an eigenvector of $L_{N}$ associated with $\mu$  iff $(v_1,v_c)$ is an eigenvector of $L_{N_1}$ associated with $\mu$.
Furthermore, $v_2$ can be chosen such that $(0,v_2)$ is not an eigenvector of $L_{N_2}$ associated with $\mu$. \\
(ii) $\mu$ is a semisimple eigenvalue of $L_{N}$ iff $\mu$ is a semisimple eigenvalue of $L_{N_1}$ and $L_{N_2}$. 
\end{cor}

\begin{proof}
Consider the block structure of $L_N - \mu I$, in (\ref{eq:block_k}) for $k=1$. Since the first $n_1-1$ columns of $A_1$ are zero, we have that 
$$(L_N - \mu I)(v_1,v_c,v_2)=
\begin{bmatrix}
(L_{N_1} - \mu I)(v_1,v_c)\\
A_1(v_1,v_c)+ (\underline{\overline{L}}_{N_2} - \mu I)v_2
\end{bmatrix}=
 \begin{bmatrix}
(L_{N_1} - \mu I)(v_1,v_c)\\
v_c L_{N_2}^c+ (\underline{\overline{L}}_{N_2} - \mu I)v_2
\end{bmatrix}.$$
It follows from $L_{N_2}^c\in \Im((\underline{\overline{L}}_{N_2} - \mu I))$, that there exists $v_2$ such that $(\underline{\overline{L}}_{N_2} - \mu I)v_2=-v_cL_{N_2}^c$. Taking that $v_2$, we have that 
$$(L_N - \mu I)(v_1,v_c,v_2)=0 \iff (L_{N_1} - \mu I)(v_1,v_c)=0.$$
This proves the equivalence in the first item (i). 
If $v_c\neq 0$, then $v_2$ is not an eigenvector of $L_{N_2}$ associated with $\mu$, because the first row of $L_{N_2}$ is zero and $\mu\neq 0$. 
If $v_c= 0$, as in Proposition~\ref{prop:eigevaluenoninter}, we can take $v_2=0$ which is not an eigenvector of $L_{N_2}$ associated with $\mu$. This proves the first part of the result (i).

From (i) and Proposition~\ref{prop:eig_vec}, we have that $m_g(\mu) = m^1_g(\mu) + m^2_g(\mu)$ and we know from Proposition~\ref{prop:eig_val}~(ii) that $m_a(\mu) = m^1_a(\mu) + m^2_a(\mu)$. This proves (ii).
\end{proof}

\begin{rem} \label{rmk:geneig}
Assume that $\mu\neq 0$ is an eigenvalue of $L_{N}$, $L_{N_2}^c\notin \Im((\underline{\overline{L}}_{N_2} - \mu I))$  and $(v_1,v_c)$, with $v_c\neq 0$, is an eigenvector of $L_{N_1}$ associated with $\mu$. 
Then, $(v_1,v_c,v_2)$ is a generalized eigenvector of $L_{N}$ associated with $\mu$, for some $v_2\in \RR^{n_2-1}$, such that $v_2$ is not an eigenvector of $\underline{\overline{L}}_{N_2}$.
Furthermore, the eigenvalue $\mu$ is not semisimple for $L_{N}$, regardless of being semisimple for $L_{N_1}$ and $L_{N_2}$ or not.
\end{rem}

Next, we analyze some examples that illustrate the previous results.
In the first example, we consider the network $6$ studied in ~\cite{ADGL09}. 
In this case the Laplacian of the component networks share a nontrivial semisimple eigenvalue ($\mu\neq0$).
Moreover, the outgoing connections from the merging cell to $N_2$ belong to the appropriate image.
So the semisimplicty of that eigenvalue is preserved in the Laplacian of the coalescence network as stated in Corollary~\ref{cor:semisimpleeigen}.

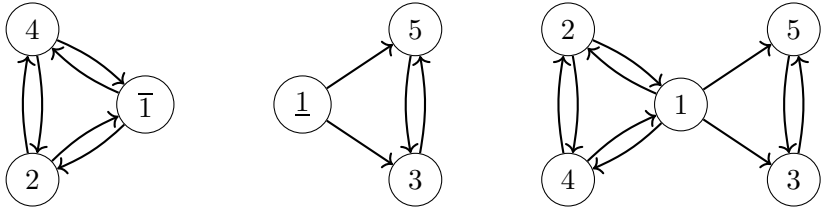
\begin{figure}[!h]
\begin{center}
\begin{tabular}{ccc}
\begin{tikzpicture}
 [scale=.25,auto=left, node distance=1.5cm]
 \node[style={circle,draw}] (n4) at (4,4) {\small{4}};
  \node[style={circle,draw}] (n1) at (10,0) {\small{$\overline{1}$}};
 \node[style={circle,draw}] (n2) at (4,-4) {\small{2}};
  \path
 (n1) edge[->, thick, bend left=10]  node  [above=0.1pt] {} (n2)
 (n1) edge[->, thick, bend left=10]  node  [above=0.1pt] {} (n4)
 (n2) edge[->, thick, bend left=10,]  node  [below=0.1pt] {} (n1)
 (n2) edge[->, thick, bend left=10,]  node  [below=0.1pt] {} (n4)
 (n4) edge[->, thick, bend left=10,]  node  [below=0.1pt] {} (n1)
 (n4) edge[->, thick, bend left=10,]  node  [below=0.1pt] {} (n2); 
\end{tikzpicture}  \qquad & \qquad 
\begin{tikzpicture}
 [scale=.25,auto=left, node distance=1.5cm]
 \node[style={circle,draw}] (n5) at (10,4) {\small{5}};
  \node[style={circle,draw}] (n1) at (4,0) {\small{$\underline{1}$}};
 \node[style={circle,draw}] (n3) at (10,-4) {\small{3}};
  \path
 (n1) edge[->, thick, bend right=0]  node  [below=0.1pt] {} (n5)
 (n1) edge[->, thick, bend right=0]  node  [below=0.1pt] {} (n3)
 (n5) edge[->, thick, bend right=10,]  node  [above=0.1pt] {} (n3)
 (n3) edge[->, thick, bend right=10,]  node  [above=0.1pt] {} (n5); 
\end{tikzpicture}  \qquad & \qquad 
\begin{tikzpicture}
 [scale=.25,auto=left, node distance=1.5cm]
 \node[style={circle,draw}] (n2) at (4,4) {\small{2}};
  \node[style={circle,draw}] (n1) at (10,0) {\small{$1$}};
 \node[style={circle,draw}] (n4) at (4,-4) {\small{4}};
 \node[style={circle,draw}] (n5) at (16,4) {\small{5}};
 \node[style={circle,draw}] (n3) at (16,-4) {\small{3}};
  \path
  (n1) edge[->, thick, bend left=10]  node  [above=0.1pt] {} (n2)
 (n1) edge[->, thick, bend left=10]  node  [above=0.1pt] {} (n4)
 (n2) edge[->, thick, bend left=10,]  node  [below=0.1pt] {} (n1)
 (n2) edge[->, thick, bend left=10,]  node  [below=0.1pt] {} (n4)
 (n4) edge[->, thick, bend left=10,]  node  [below=0.1pt] {} (n1)
 (n4) edge[->, thick, bend left=10,]  node  [below=0.1pt] {} (n2)
 (n1) edge[->, thick, bend right=0]  node  [below=0.1pt] {} (n5)
 (n1) edge[->, thick, bend right=0]  node  [below=0.1pt] {} (n3)
 (n5) edge[->, thick, bend right=10,]  node  [above=0.1pt] {} (n3)
 (n3) edge[->, thick, bend right=10,]  node  [above=0.1pt] {} (n5); 
 \end{tikzpicture} 
\end{tabular}
\end{center}
\caption{The network $N_1 \circ N_2$  (right), network $6$ in ~\cite{ADGL09}, is the coalescence of the two 3-cell motis $N_1$ and $N_2$ (left and center), by merging cells $\overline{1}$ and $\underline{1}$.} \label{Fig:coal_3-cell_net6}
\end{figure}

\begin{exam}\label{exam:net6}
 The network $6$ in ~\cite{ADGL09}, see Figure~\ref{Fig:coal_3-cell_net6}, is the coalescence of the two 3-cell motis $N_1$ and $N_2$ with the following Laplacian matrices:
$$
L_{N_1} =
\left[
\begin{array}{ccc}
2 & -1 & -1 \\
-1 & 2  & -1 \\
-1 & -1 & 2 \\
\end{array}
\right]\, 
$$
and 
$$
L_{N_2} =
\left[
\begin{array}{ccc}
0 & 0 & 0 \\
-1 & 2  & -1 \\
-1 & -1 & 2 \\
\end{array}
\right]\, . 
$$
The Laplacian $L_{N_1}$ has eigenvalues $0, 3$ and
$$
E_0=<(1,1,1)> \quad \mbox{and} \quad E_{3}=<(1,0,-1), (0,1,-1)>.
$$
The Laplacian matrix of $N_2$ has eigenvalues $0, 1, 3$ and
$$
E_0=<(1,1,1)>, \quad E_1=<(0,1,1)> \quad \mbox{and} \quad E_3=<(0,1,-1)>.
$$

The Laplacian matrix of $N=N_1 \circ N_2$ where the third cell of $N_1$ is merged with the first cell of $N_2$ is
$$
L_{N} =
\left[
\begin{array}{ccccc}
2 & -1 & -1 & 0 & 0 \\
-1 & 2  & -1 & 0 & 0\\
-1 & -1 & 2 & 0 & 0 \\
0 & 0 & -1 & 2  & -1 \\
0 & 0 & -1 & -1 & 2 \\
\end{array}
\right]\,. 
$$
It follows from Proposition~\ref{prop:eig_val} that the eigenvalues of $L_N$ are $0, 1, 3$.
Using Proposition~\ref{prop:eig_vec}, we know that
$$
E_0=<(1,1,1,1,1)>, \quad E_1=<(0,0,0,1,1)>. 
$$
Moreover, we have that $L_{N_2}^c = (-1,-1)\in \Im((\underline{\overline{L}}_{N_2} - 3I))$. By Corollary~\ref{cor:semisimpleeigen}, we know that the eigenvalue $3$ is also semisimple for $L_N$. Concretely, we have that  
\[E_3=<(1,0,-1,0,1), (0,1,-1,0,1), (0,0,0,1,-1)>. \qedhere\] %\hfill \Diamond$$
\end{exam}

In the second example the Laplacians of the component networks $N_1$ and  $N_2$ share a nontrivial eigenvalue ($\mu\neq 0$) which is simple for both of them.
In this case the outgoing connections from the merging cell to $N_2$ do not belong to the appropriate image.
Following Remark~\ref{rmk:geneig}, that eigenvalue is degenerated for Laplacian of  the coalescence network $N_1 \circ N_2$. 

\begin{figure}[!h]
\begin{center}
\begin{tabular}{ccc}
\begin{tikzpicture}
 [scale=.25,auto=left, node distance=1.5cm]
 \node[style={circle,draw}] (n2) at (10,4) {\small{2}};
  \node[style={circle,draw}] (n1) at (4,0) {\small{1}};
 \node[style={circle,draw}] (n3) at (10,-4) {\small{$\overline{3}$}};
  \path
  (n1) edge[->, thick]  node  [above=0.1pt] {} (n2)
 (n1) edge[->, thick]  node  [above=0.1pt] {} (n3)
 (n2) edge[->, bend left=10, thick]  node  [below=0.1pt] {} (n3)
 (n3) edge[->, bend left=10, thick]  node  [below=0.1pt] {} (n2); 
\end{tikzpicture}  \qquad & \qquad
\begin{tikzpicture}
 [scale=.25,auto=left, node distance=1.5cm]
 \node[style={circle,draw}] (n4) at (10,4) {\small{4}};
  \node[style={circle,draw}] (n3) at (4,0) {\small{$\underline{3}$}};
 \node[style={circle,draw}] (n5) at (10,-4) {\small{5}};
  \path
%(n4) edge [loop above, thick] node {} (n4)  
(n3) edge[->, thick]  node  [above=0.1pt] {} (n4)
 (n3) edge[->, thick]  node  [above=0.1pt] {} (n5)
  (n4) edge[->, bend left=10, thick]  node  [below=0.1pt] {} (n5)
	 (n5) edge[->, bend left=10, thick]  node  [left=0.1pt] {{\footnotesize $2$}} (n4); 
\end{tikzpicture}  \qquad & \qquad
\begin{tikzpicture}
 [scale=.25,auto=left, node distance=1.5cm]
 \node[style={circle,draw}] (n2) at (10,4) {\small{2}};
  \node[style={circle,draw}] (n1) at (4,0) {\small{1}};
 \node[style={circle,draw}] (n3) at (10,-4) {\small{3}};
 \node[style={circle,draw}] (n4) at (16,4) {\small{4}};
 \node[style={circle,draw}] (n5) at (16,-4) {\small{5}};
  \path
   (n1) edge[->, thick]  node  [above=0.1pt] {} (n2)
 (n1) edge[->, thick]  node  [above=0.1pt] {} (n3)
 (n2) edge[->, bend left=10, thick]  node  [below=0.1pt] {} (n3)
 (n3) edge[->, bend left=10, thick]  node  [below=0.1pt] {} (n2)
%(n4) edge [loop above, thick] node {} (n4)  
(n3) edge[->, thick]  node  [above=0.1pt] {} (n4)
 (n3) edge[->, thick]  node  [above=0.1pt] {} (n5)
  (n4) edge[->, bend left=10, thick]  node  [below=0.1pt] {} (n5)
	(n5) edge[->, bend left=10, thick]  node  [left=0.1pt] {{\footnotesize $2$}} (n4); 
 \end{tikzpicture} 
\end{tabular}
\end{center}
\caption{The network $N_1 \circ N_2$  (right) is the coalescence of the two 3-cell motis $N_1$ and $N_2$ (left and center), by merging cells $\overline{3}$ and $\underline{3}$.} \label{Fig:ex_gener_4}
\end{figure}

\begin{exam}\label{exam:ex_gener_4}
Consider the networks $N_1,  N_2$ and $N=N_1 \circ N_2$ in Figure~\ref{Fig:ex_gener_4}. The Laplacian matrix of $N_1$ 
$$
L_{N_1}  =
\left[
\begin{array}{ccc}
0 & 0 & 0 \\
-1 & 2  & -1 \\
-1 & -1 & 2 \\
\end{array}
\right]\, 
$$
has eigenvalues $0, 1, 3$ with
$$
E_0=<(1,1,1)>, \quad E_1=<(0,1,1)> \quad \mbox{and} \quad E_3=<(0,1,-1)>.
$$

The Laplacian matrix of $N_2$ 
$$
L_{N_2}  =
\left[
\begin{array}{ccc}
0 & 0 & 0 \\
-1 & 3  & -2 \\
-1 & -1 & 2 \\
\end{array}
\right]\, 
$$
has eigenvalues $0, 1, 4$ with
$$
E_0=<(1,1,1)>, \quad E_1=<(0,1,1)> \quad \mbox{and} \quad E_4=<(0,-2,1)>.
$$

It follows from Propositions~\ref{prop:eig_val} and ~\ref{prop:eig_vec} that  the Laplacian matrix of $N=N_1 \circ N_2$ 
$$
L_{N} =
\left[
\begin{array}{ccccc}
0 & 0 & 0 & 0 & 0 \\
-1 & 2  & -1 & 0 & 0\\
-1 & -1 & 2 & 0 & 0 \\
0 & 0 & -1 & 3  & -2 \\
0 & 0 & -1 & -1 & 2 \\
\end{array}
\right]\, 
$$
has eigenvalues $0, 1,3,4$, with
$$ E_0=<(1,1,1,1,1)>, \quad E_3=<(0,1,-1,1/2,1/2)>, \quad E_4=<(0,0,0,-2,1)>. $$

Moreover, given the eigenvector $(0,1,1)$ of $L_{N_2}$, it follows that $(0,0,0,1,1)$ is an eigenvector of $L_N$ associated with $1$. 
Note that $L_{N_2}^c =(-1,-1)\notin \Im((\underline{\overline{L}}_{N_2} - I))$ and the eigenvector $(0,1,1)$ of $L_{N_1}$ corresponds to the generalized eigenvector $(0,1,1,-1,-1)$ of $L_N$, according with Remark~\ref{rmk:geneig}.
So
\[
E_1=<(0,0,0,1,1)> \mbox{and} \quad G_1=<(0,0,0,1,1), (0,1,1,-1,-1)>. \qedhere\]
\end{exam}

In the following example, the Laplacians of the component networks $N_1$ and $N_2$ share a nontrivial eigenvalue ($\mu\neq 0$) which is degenerated for $N_2$. 
The outgoing connection from the merging cell to $N_2$ do not belong to the apropriate image
%The coalescence network $N_1 \circ N_2$ also has that degenerated eigenvalue, however, the degeneracy of that eigenvalue increases.  
and the degeneracy of that eigenvalue increases for the Laplacian of the coalescence network $N_1 \circ N_2$.

\begin{figure}[!h]
\begin{center}
\begin{tabular}{ccc}
\begin{tikzpicture}
 [scale=.25,auto=left, node distance=1.5cm]
 \node[style={circle,draw}] (n2) at (10,4) {\small{2}};
  \node[style={circle,draw}] (n1) at (4,0) {\small{1}};
 \node[style={circle,draw}] (n3) at (10,-4) {\small{$\overline{3}$}};
  \path
  (n1) edge[->, thick]  node  [above=0.1pt] {} (n2)
 (n1) edge[->, thick]  node  [above=0.1pt] {} (n3)
(n2) edge [loop above, thick] node {} (n2)   
 (n2) edge[->, thick]  node  [below=0.1pt] {} (n3); 
\end{tikzpicture}  \qquad & \qquad
\begin{tikzpicture}
 [scale=.25,auto=left, node distance=1.5cm]
 \node[style={circle,draw}] (n4) at (10,4) {\small{4}};
  \node[style={circle,draw}] (n3) at (4,0) {\small{$\underline{3}$}};
 \node[style={circle,draw}] (n5) at (10,-4) {\small{5}};
  \path
(n3) edge [loop above, thick] node {{\footnotesize $2$}} (n3)  
%  (n3) edge[->,bend left=10, thick]  node  [above=0.1pt] {} (n4)
%(n3) edge[->, thick,bend right=10]  node  [above=0.1pt] {} (n4)
(n3) edge[->, thick]  node  [above=0.1pt] {{\footnotesize $2$}} (n4)
 (n3) edge[->, thick]  node  [above=0.1pt] {} (n5)
  (n4) edge[->, thick]  node  [below=0.1pt] {} (n5); 
\end{tikzpicture}  \qquad & \qquad
\begin{tikzpicture}
 [scale=.25,auto=left, node distance=1.5cm]
 \node[style={circle,draw}] (n2) at (10,4) {\small{2}};
  \node[style={circle,draw}] (n1) at (4,0) {\small{1}};
 \node[style={circle,draw}] (n3) at (10,-4) {\small{3}};
 \node[style={circle,draw}] (n4) at (16,4) {\small{4}};
 \node[style={circle,draw}] (n5) at (16,-4) {\small{5}};
  \path
  (n1) edge[->, thick]  node  [above=0.1pt] {} (n2)
 (n1) edge[->, thick]  node  [above=0.1pt] {} (n3)
(n2) edge [loop above, thick] node {} (n2)   
 (n2) edge[->, thick]  node  [below=0.1pt] {} (n3)
(n3) edge [loop left, thick] node {{\footnotesize $2$}} (n3)  
(n3) edge[->, thick]  node  [above=0.1pt] {{\footnotesize $2$}} (n4)
 (n3) edge[->, thick]  node  [above=0.1pt] {} (n5)
  (n4) edge[->, thick]  node  [below=0.1pt] {} (n5); 
 \end{tikzpicture} 
\end{tabular}
\end{center}
\caption{The network $N_1 \circ N_2$  (right) is the coalescence of the two 3-cell motis $N_1$ and $N_2$ (left and center), by merging cells $\overline{3}$ and $\underline{3}$.} \label{Fig:ex_gener_1}
\end{figure}

\begin{exam}\label{exam:ex_gener_1}
Consider the networks $N_1,  N_2$ and $N=N_1 \circ N_2$ in Figure~\ref{Fig:ex_gener_1}. The Laplacian matrix of $N_1$ 
$$
L_{N_1}  =
\left[
\begin{array}{ccc}
0 & 0 & 0 \\
-1 & 1  & 0 \\
-1 & -1 & 2 \\
\end{array}
\right]\, 
$$
has eigenvalues $0, 1, 2$ with
$$
E_0=<(1,1,1)>, \quad E_1=<(0,1,1)> \quad \mbox{and} \quad E_2=<(0,0,1)>.
$$

The Laplacian matrix of $N_2$ 
$$
L_{N_2}  =
\left[
\begin{array}{ccc}
0 & 0 & 0 \\
-2 & 2  & 0 \\
-1 & -1 & 2 \\
\end{array}
\right]\, 
$$
has eigenvalues $0, 2$ with
$$
E_0=<(1,1,1)>, \quad E_2=<(0,0,1)> \quad \mbox{and} \quad G_2=<(0,0,1), (0,-1,1)>.
$$

It follows from Proposition~\ref{prop:eig_val} that, the Laplacian matrix of $N=N_1 \circ N_2$, 
$$
L_{N} =
\left[
\begin{array}{ccccc}
0 & 0 & 0 & 0 & 0 \\
-1 & 1  & 0 & 0 & 0\\
-1 & -1 & 2 & 0 & 0 \\
0 & 0 & -2 & 2  & 0 \\
0 & 0 & -1 & -1 & 2 \\
\end{array}
\right]\, 
$$
has eigenvalues $0, 1, 2$.

Using Corollary~\ref{Rems:eigenvec_0}, we have that 
$$ E_0=<(1,1,1,1,1)>.$$

Since $1$ is not an eigenvalue of $L_{N_2}$ and $\underline{\overline{L}}_{N_2}$, we know that $L_{N_2}^c =(-2,-1)\in \Im((\underline{\overline{L}}_{N_2} - I))$. It follows, from Corollary~\ref{cor:semisimpleeigen}, that the eigenvalue $1$ is semisimple for $L_N$ and there exists $\hat{w}_1\in\RR^2$ such that $(0,1,1,\hat{w}_1)$ is an eigenvector of $L_N$ associated with $1$. In this case, we have $$ E_1=<(0,1,1,2,3)>.$$

The eigenvector $(0,0,1)$ of $L_{N_2}$ associated with the eigenvalue $2$ vanishes in the merging coordinate. 
Using Proposition~\ref{prop:eigevaluenoninter}, we know that $(0,0,0,0,1)$ is an eigenvector of $L_N$ associated with $2$. 
Moreover, the generalized eigenvector $(0,-1,1)$ of $L_{N_2}$ associated with $2$ also vanishes in the merging coordinate. 
It follows from Proposition~\ref{prop:eig_vec}~(iv) that $(0,0,0,-1,1)$ is a generalized eigenvector of $L_N$. 
The eigenvector $(0,0,1)$ of $L_{N_1}$ corresponds to a generalized eigenvector $(0,0,1,\hat{w}_2)$ for some $\hat{w}_2\in\RR^2$, because $L_{N_2}^c =(-2,-1)\notin \Im((\underline{\overline{L}}_{N_2} - 2I))$ and the merging coordinate of the eigenvector $(0,0,1)$ is nonzero. 
In this example $\hat{w}_2=(-3,0)$,
\[ E_2=<(0,0,0,0,1)>
\mbox{and} \quad G_2=<(0,0,0,0,1), (0,0,0,-1,1), (0,0,1,-3,0)>.\qedhere\]
%\hspace{2mm} \Diamond$$
\end{exam}

In the last example, the component networks $N_1$ and $N_2$ have the same eigenvalues, which are semisimple, and the coalescence network preserves the semisimplicity of the eigenvalues. 
In this case the outgoing connections from the merging cell to $N_2$ do not belong to the appropriate image, however the eigenvectors associated with the nontrivial eigenvalue ($\mu\neq 0$) vanish at the merging cell.

\begin{figure}[!h]
\begin{center}
\begin{tabular}{ccc}
\begin{tikzpicture}
 [scale=.25,auto=left, node distance=1.5cm]
 \node[style={circle,draw}] (n1) at (4,4) {\small{1}};
  \node[style={circle,draw}] (n3) at (10,0) {\small{2}};
 \node[style={circle,draw}] (n2) at (4,-4) {\small{$\overline{3}$}};
  \path
(n1) edge [loop above, thick] node {{\footnotesize $2$}} (n1)   
 (n1) edge[->, thick]  node  [above=0.1pt] {} (n3)
(n2) edge [loop above, thick] node {{\footnotesize $2$}} (n2)   
 (n2) edge[->, thick]  node  [below=0.1pt] {} (n3); 
\end{tikzpicture}  \qquad & \qquad

\begin{tikzpicture}
 [scale=.25,auto=left, node distance=1.5cm]
 \node[style={circle,draw}] (n1) at (4,4) {\small{$\underline{3}$}};
  \node[style={circle,draw}] (n3) at (10,0) {\small{$4$}};
 \node[style={circle,draw}] (n2) at (4,-4) {\small{$5$}};
  \path
(n1) edge [loop above, thick] node {{\footnotesize $2$}} (n1)   
 (n1) edge[->, thick]  node  [above=0.1pt] {} (n3)
(n2) edge [loop above, thick] node {{\footnotesize $2$}} (n2)   
 (n2) edge[->, thick]  node  [below=0.1pt] {} (n3); 
\end{tikzpicture}  \qquad & \qquad 
\begin{tikzpicture}
 [scale=.25,auto=left, node distance=1.5cm]
 \node[style={circle,draw}] (n1) at (4,4) {\small{1}};
  \node[style={circle,draw}] (n2) at (7,-4) {\small{$2$}};
 \node[style={circle,draw}] (n3) at (10,4) {\small{3}};
 \node[style={circle,draw}] (n4) at (13,-4) {\small{4}};
 \node[style={circle,draw}] (n5) at (16,4) {\small{5}};
  \path
(n1) edge [loop above, thick] node {{\footnotesize $2$}} (n1)   
 (n1) edge[->, thick]  node  [above=0.1pt] {} (n2)
%(n2) edge [loop above, thick] node {{\footnotesize $2$}} (n2)   
 %(n2) edge[->, thick]  node  [below=0.1pt] {} (n3)
(n3) edge[->, thick]  node  [above=0.1pt] {} (n2)
(n3) edge [loop above, thick] node {{\footnotesize $4$}} (n3)   
(n3) edge[->, thick]  node  [above=0.1pt] {} (n4)
(n5) edge[->, thick]  node  [below=0.1pt] {} (n4) 
(n5) edge [loop above, thick] node {{\footnotesize $2$}} (n5) ;  
 \end{tikzpicture} 
\end{tabular}
\end{center}
\caption{The network $N_1 \circ N_2$  (right) is the coalescence of the two 3-cell motis $N_1$ and $N_2$ (left and center), by merging cells $\overline{3}$ and $\underline{3}$.} \label{Fig:ex_ma(0)>1}
\end{figure}

\begin{exam}\label{exam:ex_ma(0)>1} 
%\footnote{MA: Este exemplo \'e interessante porque tanto encaixa nesta sec\c{c}\~ao das FFCNs como na sec\c{c}\~ao seguinte das SCINs. PS: a rede obtida é ligeramente diferente no número de self-loop que a celula $3$ recebe. No caso das SCIN continua a receber apenas dois self-inputs.}
Consider the networks $N_1,  N_2$ and $N=N_1 \circ N_2$ in Figure~\ref{Fig:ex_ma(0)>1}. The Laplacian matrices of $N_1$ and $N_2$ are equal
$$
L_{N_1} =L_{N_2} =
\left[
\begin{array}{ccc}
0 & 0 & 0 \\
-1 & 2  & -1 \\
0 & 0 & 0 \\
\end{array}
\right]\, 
$$
and have eigenvalues $0, 2$ with
$$
E_0=<(1,1,1),  (0,1,2)> \quad \mbox{and} \quad E_2=<(0,1,0)>.
$$

According to Proposition~\ref{prop:eig_val}, the Laplacian matrix of $N=N_1 \circ N_2$, 
$$
L_{N} =
\left[
\begin{array}{ccccc}
0 & 0 & 0 & 0 & 0 \\
-1 & 2  & -1 & 0 & 0\\
0 & 0 & 0 & 0 & 0 \\
0 & 0 & -1 & 2  & -1 \\
0 & 0 & 0 & 0 & 0 \\
\end{array}
\right]\, 
$$
has the eigenvalues $0, 2$.

Using Corollary~\ref{Rems:eigenvec_0}, we know that there exist $\hat{w}_1, \hat{w}_2\in\RR^2$ such that $(1,1,1,\hat{w}_1)$, $(0,1,2,\hat{w}_2)$ and $(0,0,0,1,2)$ is a basis of the eigenspace $E_0$ for $L_{N}$.
In this example, $\hat{w}_1=(1,1)$ and $\hat{w}_2=(0,-2)$. Thus
$$
E_0=<(1,1,1,1,1), (0,1,2,0,-2), (0,0,0,1,2)>$$

The eigenvectors associated with the eigenvalue $2$ have the coalescence coordinate equal to zero. So, we can apply Proposition~\ref{prop:eigevaluenoninter} to conclude that $(0,1,0,0,0)$ and $(0,0,0,1,0)$ are eigenvectors of $L_N$ associated with the eigenvalue $2$. Thus
\[ E_2=<(0,1,0,0,0), (0,0,0,1,0)>.\]
In this case we have that
$L_{N_2}^c = (-1,0)\notin \Im(\left[
\begin{array}{ccccc}
0 & -1 \\
0 & -2 \\
\end{array}
\right])$ .
Despite the assumption in Corollary~\ref{cor:semisimpleeigen} does not hold, however the eigenvalues in the networks $N_1$ and $N_2$ associated with the eigenvalue $2$ vanish in the merging coordinate. So the eigenvalue remains semisimple in the coalescence network. 
\end{exam}

Propositions~\ref{prop:eig_val} and \ref{prop:eig_vec} generalize for feedforward coalescence networks given by the coalescence of more than two networks. 
In fact, given such a network $N = N_1 \circ N_2 \circ \cdots \circ N_r$, with $r \geq 1$, we can see it, recursively, as $N = M_{r-1} \circ N_r$ by merging the last cell in $M_{r-1}$ with the first cell in $N_r$, with $M_{1}= N_1$.
Thus a recursive application of Propositions~\ref{prop:eig_val} and \ref{prop:eig_vec} leads to the following result. 

%The Laplacian matrix for $N$ has the following block structure
%\begin{equation} \label{eq:block_L_geral}
%L_N=
%\left[
%\begin{array}{cccc}
%L_{N_1} & | 0_{n_1 \times (n_2-1)} & | \cdots & | 0_{n_1 \times (n_r-1)} \\
%\hline
%0_{(n_2-1) \times (n_1-1)} | & \overline{L}_{N_2} & | \cdots & | 0_{(n_2-1) \times (n_r-1)} \\
%\hline
%\cdots & | \cdots & | \cdots & | \cdots \\
%\hline
%0_{(n_r-1) \times (n_1-1)} && | & \overline{L}_{M_r}
%\end{array}
%\right]\, .
%\end{equation}

\begin{prop} \label{prop:eig_geral}
Let $N = N_1 \circ N_2 \circ \cdots \circ N_r$ be a FFCN where the networks $N_1, N_2, \cdots, N_r$ have $n_1,n_2,\cdots,n_r$ cells, respectively.

\vspace{0.1in}

\noindent [1] The eigenvalues of $L_N$ are the union of the eigenvalues of $L_{N_i}$, $i=1, \ldots, r$, with \\
\noindent  (i)  $m_a(0) = \sum_{i=1}^r m^i_a(0) -r + 1$; \\
\noindent (ii) $m_a(\mu) =\sum_{i=1}^r m^i_a(\mu)$, for $\mu \ne 0$.

\vspace{0.1in}

\noindent [2] For each eigenvalue $\mu$ of $L_{N}$, a (generalized) eigenvector of $L_{N}$ associated with $\mu$ is of the form $(v_1, w_2, \ldots, w_r)$,
$(0_{{\footnotesize\R^{n_1 + \ldots + n_{i-1}-i+1}}}, v_i, w_{i+1}, \ldots, w_r)$, with $i=2, \ldots, r-1$,
or $(0_{{\footnotesize\R^{n_1 + \ldots + n_{r-1}-r+1}}}, v_r)$.

\vspace{0.05in}
\noindent For $\mu  \ne 0$,\\
\noindent (i) $(v_1, w_2, \ldots, w_r)$ is a generalized eigenvector of $L_{N}$ associated with $\mu$, for some $w_j \in \R^{n_j-1}$, $j=2,\ldots,r$, iff $v_1$ is a (generalized) eigenvector of $L_{N_1}$ associated with $\mu$;

\noindent (ii) $(0_{{\footnotesize\R^{n_1 + \ldots + n_{i-1}-i+1}}}, v_i, w_{i+1}, \ldots, w_r)$, with $i=2, \ldots, r-1$, is a generalized eigenvector of $L_{N}$ associated with $\mu$ for some $w_j \in \R^{n_j-1}$, $j=i+1,\ldots,r$,  iff $v_i$ is a (generalized) eigenvector of $L_{N_i}$ associated with $\mu$;

\noindent (iii) $(0_{{\footnotesize\R^{n_1 + \ldots + n_{r-1}-r+1}}}, v_r)$ is a (generalized) eigenvector of $L_{N}$ associated with $\mu$ iff $v_r$ is (generalized) eigenvector of $L_{N_r}$ associated with $\mu$.

\vspace{0.05in}
\noindent For $\mu  = 0$, \\
\noindent (iv) $(v_1, w_2, \ldots, w_r)$ is an eigenvector of $L_{N}$ associated with the eigenvalue $0$ iff $v_1$ is an eigenvector of $L_{N_1}$, $(v_{1,n_1}, w_2)$ is an eigenvector of $L_{N_2}$ and $(w_{j-1,n_{j-1}}, w_j)$ is an eigenvector of $L_{N_j}$, $j=3, \ldots,r$ associated with the eigenvalue $0$;

\noindent (v) $(0_{{\footnotesize\R^{n_1 + \ldots + n_{i-1}-i+1}}}, v_i, w_{i+1}, \ldots, w_r)$, with $i=2, \ldots, r-1$, is an eigenvector of $L_{N}$ associated with the eigenvalue $0$ iff $v_i$ is an eigenvector of $L_{N_i}$ associated with the eigenvalue $0$ and such that $v_{i,1}=0$  and $(w_{j-1,n_{j-1}}, w_j)$ is an eigenvector of $L_{N_j}$ associated with the eigenvalue $0$, $j=i+1, \ldots,r$;

\noindent (vi) $(0_{{\footnotesize\R^{n_1 + \ldots + n_{r-1}-r+1}}}, v_r)$ is an eigenvector of $L_{N}$ associated with the eigenvalue $0$ iff $v_r$ is an eigenvector of $L_{N_r}$ associated with the eigenvalue $0$ and such that $v_{r,1}=0$.
\end{prop}

\subsection{The bifurcation coalescence problem for FFCNs}

Now, we focus on the steady-state bifurcations of diffusive admissible systems of feedforward coalescence networks $N=N_1\circ N_2$.
Recall that the bifurcation conditions are given by $g_0 + \mu {h_1}=0$, where $\mu$ is an eigenvalue of the Laplacian matrix $L_N$. 
Moreover, Proposition~\ref{prop:eig_val} asserts that the bifurcation conditions for $N=N_1\circ N_2$ are also bifurcation conditions for $N_1$ and/or $N_2$.

We start by considering the case where the bifurcation condition holds for exactly one of the component networks.
When the bifurcation condition holds for the network $N_1$ ($N_2$) and does not hold for the network $N_2$ ($N_1$), 
we have a unique correspondence between the bifurcation branches for the network $N_1$ ($N_2$)  and the bifurcation branches for $N$.

\begin{prop}\label{Prop:bif_cond_one}
Let $N = N_1 \circ N_2$ be a feedforward coalescence network and $F^N$ be an $N$-admissible diffusive system.
Given a bifurcation condition for $F^{N}$ that is a bifurcation condition for $F^{N_1}$, but not for $F^{N_2}$, each bifurcation branch of $F^{N_1}$ corresponds to a unique bifurcation branch of $F^{N}$, and vice-versa.
The analogue is valid if it is a bifurcation condition for $F^{N_2}$, but not for $F^{N_1}$.
\end{prop}
\begin{proof}
It follows from the feedforward structure of $N$, see Remark~\ref{rem:admcolsys}, and the fact of the admissible system being diffusive that 
$$F^N(x_1,x_c,x_2,\lambda)=0 \iff F^{N_1}(x_1,x_c,\lambda)=0 \wedge F^{\overline{N_2}}(x_c,x_2,\lambda)=0,$$
where $\overline{N_2}$ is the network $N_2$ without the cell $c$.

Suppose that we have a bifurcation condition for $N_1$ that is not a bifurcation condition for $N_2$. 
Then the solutions of $F^{N_1}(x_1,x_c,\lambda)=0$ are given by the bifurcation branches of $N_1$.
Consider $(b_1,b_c)$ a bifurcation branch of $N_1$. 
Now, we need to solve $f^{\overline{N_2}}(b_c(\lambda),x_2,\lambda)=0$. 
Since the bifurcation condition does not hold for $N_2$, it does not hold for $\overline{N_2}$ and there is a unique solution to $f^{\overline{N_2}}(b_c(\lambda),x_2,\lambda)=0\iff x_2=b_2(\lambda)$.
Thus, there exists a unique correspondence between bifurcation branches of $N_1$ and the bifurcation branches of $N$.  
 
Now, consider the case where the bifurcation condition for $N_2$ is not a bifurcation condition for $N_1$
and note that the unique equilibrium for $N_1$ in a neighbourhood of the origin is the origin.
Since $N$ is a FFCN, the merging cell does not receive inputs from the other cells in $N_2$. Let $g_0$ be the eigenvalue of the Jacobian matrix of $N_2$ associated to the merging cell.
If the merging cell bifurcates then the bifurcation condition must be given by ${g_0}=0$ which is also a bifurcation condition for $N_1$.
Thus, given the assumption that the bifurcation condition holds for $N_2$ and does not hold for $N_1$, we conclude that the merging cell does not bifurcate.
We have then that the bifurcation branches in $N$ are given by the coalescence of the zero equilibrium branch in $N_1$ and the bifurcation branches in $N_2$.  
\end{proof}

From now on we assume that the bifurcation condition holds for both component networks $N_1$ and $N_2$. 
For FFCNs, contrary to most coalescence networks,
there are bifurcation branches for the feedforward coalescence network $N_1 \circ N_2$
that are `the coalescence' of bifurcation branches for the networks $N_1$ and $N_2$.
In fact, given $(b_1,b^1_c)$ and $(b^2_c,b_2)$ equilibrium branches of $F^{N_1}$ and $F^{N_2}$ such that $b^1_c\equiv b^2_c$, the coalescence of these branches is an equilibrium branch of $F^N$.
For $d\neq c$ and $d\in N_i$, we have
$$(F^N)_d(b_1(\lambda),b_c(\lambda),b_2(\lambda),\lambda)= (F^{N_i})_d(b_1(\lambda),b_c(\lambda),\lambda)=0$$
and 
$$(F^N)_d(b_1(\lambda),b_c(\lambda),b_2(\lambda),\lambda)= (F^{N_1})_d(b_1(\lambda),b_c(\lambda),\lambda)+(F^{N_2})_c(b_2(\lambda),b_c(\lambda),\lambda)-g(b_c(\lambda),\lambda)=0.$$
since $g(b_c(\lambda),\lambda)=(F^{N_2})_c(b_2(\lambda),b_c(\lambda),\lambda)$.
%Are there more bifurcation branches that the one gives by the coalescence of branches?

However, as we will see, usually the bifurcation branches for $N_1$ and $N_2$ do not need to agree in the $c$-coordinate and, in that case, there are no bifurcation branches  for $N_1 \circ N_2$ given by the coalescence of branches for $N_1$ and $N_2$.
%There is also the question of the existence of other bifurcation branches for $N_1\circ N_2$ besides the ones given by the coalescence of branches of $N_1$ and $N_2$.
Moreover, as we are going to show, apart from some specific cases, there are bifurcation branches for $N_1\circ N_2$ that are not obtained by the coalescence of branches of $N_1$ and $N_2$.
However, in the next example we present a particular case where every bifurcation branch for $N_1\circ N_2$ is given by the coalescence of bifurcation branches for $N_1$ and $N_2$.

\begin{exam}
Consider the networks $N_1$, $N_2$  and $N= N_1 \circ N_2$ in Example~\ref{exam:ex_ma(0)>1} and Figure~\ref{Fig:ex_ma(0)>1} and consider the bifurcation condition given by ${g_0}+2{h_2}=0$.
Since cell $c=3$ does not receive any input from the other cells the dynamics for this cell is `isolated'.  And the dynamics for this cell is hyperbolic.
Thus the $c$-coordinate in the center manifold is always equal to zero. So, the bifurcation branches for $N$ can be obtained from the bifurcation branches for $N_1$ and $N_2$.

Knowing that the bifurcation given by ${g_0}+2{h_2}=0$ is transcritical in $N_1$ (and $N_2$) and denoting by $(0,0, 0)$ and $(0,b(\lambda), 0)$ the equilibrium branches in $N_1$ (and $N_2$).
We obtain the bifurcation branches of $N$ which are the following: 
\[(0,b(\lambda),0, 0, 0), (0,0, 0,b(\lambda), 0), (0,b(\lambda), 0,b(\lambda), 0).\qedhere\]
\end{exam}

As we saw in the proof of Proposition~\ref{Prop:bif_cond_one}, the restriction of a bifurcation branch for $N$ to the first $n_1$ cells must be a bifurcation branch for $N_1$.
Given a bifurcation branch, $b=(b_1,b_c):D_1\rightarrow \RR^{n_1-1}\times \RR$, for $N_1$, we need to solve the following equation to find the bifurcation branches for $N$
$$F^{\overline{N}_2}(b_c(\lambda),x_2,\dots,x_{n_2},\lambda)=0.$$
Let $\phi:\RR^{n_2-1}\times \RR\rightarrow\RR^{n_2-1}$ given by $$\phi(x_2,\dots,x_{n_2},\lambda)= F^{\overline{N}_2}(b_c(\lambda),x_2,\dots,x_{n_2},\lambda).$$ 
The derivative of $\phi$ with respect to the variables $x_2,\dots,x_{n_2}$ at $x_c=x_2=\dots=x_{n_2}=0$ and $\lambda=0$ is equal to the submatrix $\underline{\overline{J}}_F^{N_2}$ of $J_F^{N_2}$ by removing the $c$-column and $c$-row, $$D_{x}\phi=\underline{\overline{J}}_F^{N_2}.$$

%\begin{rem}
%If the bifurcation condition for $N_1$ is not a bifurcation condition for $N_2$, then $J_F^{N_2}$ has rank $n_2$. 
%The matrix $\underline{\overline{J}}_F^{N_2}=D_{x}\phi$ has rank $n_2-1$, and there exists a unique solution branch to the equation $\phi=0$. 
%Let $b_2:D_2\rightarrow\RR^{n_2-1}$ such that $$\phi(x_2,\dots,x_{n_2},\lambda)=0 \iff (x_2,\dots,x_{n_2})=b_2(\lambda).$$  
%Then the bifurcation branch $(b_1,b_c)$ of $N_1$ is uniquely extended to the bifurcation branch $(b_1,b_c,b_2):D=D_1\cap D_2 \rightarrow\RR^{n_1+n_2}$ of $N$, as stated in Proposition~\ref{Prop:bif_cond_one}.
%\end{rem}

The derivative of $\phi$ with respect to the variable $\lambda$ at $x_c=x_2=\dots=x_{n_2}=0$ and $\lambda=0$ is equal to 
$$D_{\lambda}\phi= (b_c)'(0)[J_F^{N_2}]^c+[(f_d)_{\lambda}]_{\substack{d\neq c,\\ d\in C_2}},$$
where $[J_F^{N_2}]^c$ is the $c$-column of $J_F^{N_2}$ without the $c$-row, and $(b_c)'(0)$ and $(f_d)_{\lambda}$ denote the derivatives of the (respective) functions in order to $\lambda$ at the origin.
Since $(f_d)_{\lambda}(0,\dots,0,\lambda)=0$ for every cell $d$, the previous derivative of $\phi$ is 
$$D_{\lambda}\phi= (b_c)'(0)[J_F^{N_2}]^c.$$

Since the bifurcation condition is shared by the two networks $N_1$ and $N_2$, then $D_{x}\phi=\underline{\overline{J}}_F^{N_2}$ is not invertible and we have
$$rank(D_{x\lambda}\phi)=\begin{cases} rank(D_{x}\phi)\quad&, \textrm{ if }  (b_c)'(0)[J_F^{N_2}]^c\in Im(\underline{\overline{J}}_F^{N_2})\\
rank(D_{x}\phi)+1&, \textrm{ if }  (b_c)'(0)[J_F^{N_2}]^c\notin Im(\underline{\overline{J}}_F^{N_2})
\end{cases}$$

We start with the simplest case, when $rank(D_{x}\phi)=n_2-2$ and $(b_c)'(0)[J_F^{N_2}]^c\notin Im(\underline{\overline{J}}_F^{N_2})$. 
In this case, there exists a cell $d$ such that the submatrix of $D_{x\lambda}\phi$, obtained by removing the $d$-column, is invertible.
By reordering the cells suppose that $d$ is the second cell in $N_2$. 
By the implicit function theorem, we know that there exists a function $z:\RR\rightarrow \RR^{n_2-2}\times \RR$ with $z(x_2)=(z_{x}(x_2),z_{\lambda}(x_2))$ such that 
$$\phi(x_2,x_3,\dots,x_{n_2},\lambda)=0\iff (x_3,\dots,x_{n_2})= z_x(x_2)\wedge \lambda=z_\lambda(x_2).$$
Then, we obtain the `branch' on the coalescence network
$$F^{N}(b_1(z_\lambda(x_2)),b_c(z_\lambda(x_2)),x_2,z_x(x_2),z_{\lambda}(x_2))=0.$$
Differentiating the function $F^{N_2}(b_c(z_\lambda(x_2)),x_2,z_x(x_2),z_{\lambda}(x_2))$ with respect to $x_2$ at $x_2=0$, we get  
%$$z_{\lambda}'(x_2) J_F^{N_1}(b_1(z_\lambda(x_2)),b_c(z_\lambda(x_2)),z_{\lambda}(x_2))\begin{bmatrix}b_1'(z_\lambda(x_2))\\ b_c'(z_\lambda(x_2))\end{bmatrix}=0 $$
%$$z_{\lambda}'(0) J_F^{N_1}\begin{bmatrix}b_1'(0)\\ b_c'(0)\end{bmatrix}=0 $$
%and 
%$$\underline{\overline{J}}_F^{N_2}(b_c(z_\lambda(x_2)),x_2,z_x(x_2),z_{\lambda}(x_2))\begin{bmatrix}1\\ z'_x(x_2)\end{bmatrix}+ [J_F^{N_2}]^c(b_c(z_\lambda(x_2)),x_2,z_x(x_2),z_{\lambda}(x_2)) (b_c)'(z_\lambda(x_2)) z'_{\lambda}(x_2)=0 $$ 
$$\underline{\overline{J}}_F^{N_2}\begin{bmatrix}1\\ z'_x(0)\end{bmatrix}+(b_c)'(0) z'_{\lambda}(0) [J_F^{N_2}]^c =0.$$ 
Since $(b_c)'(0)[J_F^{N_2}]^c\notin Im(\underline{\overline{J}}_F^{N_2})$, we know that $z'_{\lambda}(0)$ must be $0$. 
The vector $(1,z_x'(0))$ is know since it is an eigenvector of $\underline{\overline{J}}_F^{N_2}$ associated with the bifurcation condition $g_x+\mu h_1=0$ and it is also an eigenvector of $\underline{\overline{L}}_F^{N_2}$ associated with $\mu$. 
Differentiating twice the function $F^{N_2}(b_c(z_\lambda(x_2)),x_2,z_x(x_2),z_{\lambda}(x_2))$ with respect to $x_2$ at $x_2=0$, we have that
$$\dfrac{\partial^2 f_j^{N_2}}{\partial x_2^2}= \dfrac{\partial f_j}{\partial x_c}b'_c z''_{\lambda}+ \dfrac{\partial^2 f_j}{\partial x_2^2}+ 2 \sum_{i\geq 3}\dfrac{\partial^2 f_j}{\partial x_2\partial x_i}z'_{x_i}+ \sum_{i,k\geq 3} \dfrac{\partial^2 f_j}{\partial x_i\partial x_k}z'_{x_i}z'_{x_k}+ \sum_{i\geq 3} \dfrac{\partial f_j}{\partial x_i}z''_{x_i}=0,$$
for $j\in C_2\setminus\{c\}$ since $z'_{\lambda}=0$ and $(f_d)_{\lambda}=0$. 
The second derivative can be rewritten as 
$$\dfrac{\partial^2 F^{N_2}}{\partial x_2^2}= z''_{\lambda}(0) (b_c)'(0)[J_F^{N_2}]^c+ \underline{\overline{J}}_F^{N_2}\begin{bmatrix}0\\ z''_x(0)\end{bmatrix} 
+ \left[\begin{bmatrix}1 & z'_x(0)\end{bmatrix} \underline{\overline{H}}^{N_2}_{f_j} \begin{bmatrix}1 \\ z'_x(0)\end{bmatrix} \right]_{j\in C_2\setminus\{c\}} ,$$
where $\underline{\overline{H}}^{N_2}_{f_j}$ is the submatrix obtained from the Hessian matrix of $f_j^{N_2}$ by removing the $c$-column and $c$-row. 
Denote by $H$ the column matrix given by the third term of the right side of the equality above. 
The vector $H$ is known since $(1,z'_x(0))$ is an eigenvector of $\underline{\overline{L}}_F^{N_2}$ associated with $\mu$.
 And we have the equation 
$$z''_{\lambda}(0) (b_c)'(0)[J_F^{N_2}]^c+ \underline{\overline{J}}_F^{N_2}\begin{bmatrix}0\\ z''_x(0)\end{bmatrix} 
+ H=0.$$
The previous equation can be solved in order to  $z''_{x}(0)$ and $z''_{\lambda}(0)$ because we can split the space $\RR^{n_2-1}= Im(\underline{\overline{J}}_F^{N_2})\oplus <(b_c)'(0)[J_F^{N_2}]^c>$.  
Moreover, 
$$z''_{\lambda}(0)\neq 0 \iff H\notin Im(\underline{\overline{J}}_F^{N_2}).$$
Since $z''_{\lambda}(0)\neq 0$, we can rewrite the branch $z$ as two bifurcation branches as functions of the parameter $\lambda$. 
These two bifurcation branches are defined exclusively for positive or negative values of $\lambda$, depending whether $z''_{\lambda}(0)$ is positive or negative, respectively. 
Furthermore, the two bifurcation branches have a growth-rate of order $\sqrt{\lambda}$ for the cells in $N_2$ where the eigenvector $(1,z_x'(0))$ does not vanish.

 Note that $[J_F^{N_2}]^c=h_1 L_{N_2}^c$ and we can assume generically that $h_1\neq 0$, because the bifurcation condition has the form $g_x+\mu h_1=0$. We have thus shown the following result.

%We summarize the previous observations in the following result.

\begin{prop}\label{prop:extbifbra}
Let $N=N_1\circ N_2$ be a feedforward coalescence network, $F^N$ be an $N$-admissible diffusive system and $(b_1,b_c):D_1\rightarrow \RR^{n_1-1}\times \RR$ be a bifurcation branch for $F^{N_1}$.
Suppose that $\underline{\overline{J}}_F^{N_2}$ has rank $n_2-2$, $(b_c)'(0) L_{N_2}^c\notin Im(\underline{\overline{J}}_F^{N_2})$ and $H\notin Im(\underline{\overline{J}}_F^{N_2})$.
Denote by $v\in\RR^{n_2-1}$ a vector in the kernel of $\underline{\overline{J}}_F^{N_2}$.

Then, generically, there exist $b^1_2,b^2_2 :D_2\rightarrow \RR^{n_2-1}$ such that $b^1=(b_1,b_c,b^1_2),b^2=(b_1,b_c,b^2_2):D_1\cap D_2\rightarrow \RR^{n_1+n_2-1}$ are bifurcation branches for $F^{N}$. 
Furthermore, for any cell $d\in N_2$ such that $v_d\neq 0$, the growth-rate of the bifurcation branches, $b^1_d$ and $b^2_d$, at cell $d$ is $\sqrt{|\lambda|}$.
\end{prop}

Before we exemplify the previous result, we observe next that the non-degenerated condition $H\notin Im(\underline{\overline{J}}_F^{N_2})$ depends generically on the network $N_2$ that we consider. This means that, given a network $N_2$, the previous condition holds for any generically $g$ an $h$ or it fails for any $g$ and $h$.

\begin{rem}
Since $h(x,x,\lambda)=0$, the second derivatives of $h$ satisfy $h_{11}+2h_{12}+h_{22}=0$, where 
$h_{ij}$ is the second derivative of $h(x_i,x_j,\lambda)$ with respect to $x_i$ and $x_j$ at the origin.
Denote by $g_{xx}$ the second derivative of $g(x,\lambda)$ with respect to $x$ twice at the origin and by $v=(1,z'_x(0))$ the eigenvector of $L_{\overline{N}_2}$ associated with $\mu$ used before.
Then, the vector $H$ can be rewritten as
$$H=g_{xx} v*v+h_{11}(L_{\overline{N}_2}v)*v+h_{22}[W_{\overline{N}_2}v^2-(W_{\overline{N}_2}v)*v],$$
where $x*y$ is the vector given by the product of the coordinates and $W_{\overline{N}_2}v^2$ is the vector $(\sum_{i=2,i\neq j}^{n_2}w_{ji} v_i^2)_{j=2,\dots,n_2}$.  
So
$$H=(g_{xx}+\mu h_{11}) v*v+h_{22}[W_{\overline{N}_2}v^2-(W_{\overline{N}_2}v)*v].$$

And $H\notin Im(\underline{\overline{J}}_F^{N_2})$ for $g$ and $h$ generic, if and only if 
$$v*v\notin  Im(\underline{\overline{J}}_F^{N_2}) \text{ or } [W_{\overline{N}_2}v^2-(W_{\overline{N}_2}v)*v] \notin  Im(\underline{\overline{J}}_F^{N_2})$$
if and only if 
$$v*v\notin  Im(h_1(L_{\overline{N}_2}-\mu Id)) \text{ or } [W_{\overline{N}_2}v^2-(W_{\overline{N}_2}v)*v] \notin  Im(h_1(L_{\overline{N}_2}-\mu Id)).$$
Since $h_1\neq 0$, the previous conditions depend exclusively on the network structure.   
\end{rem}

In the following example we apply the previous result to the feedforward coalescence network in Example~\ref{exam:ex_gener_4}.

\begin{exam}\label{exam:ex_gener_4b}
Let $N=N_1\circ N_2$ be the feedforward coalescence network given in Example~\ref{exam:ex_gener_4} and let $F^N$ be an $N$-admissible diffusive system such that $g_x+h_1=0$ and $F^{N}(0,0,0,0,0,\lambda)=0$.
In order to study the bifurcation branches for $F^N$, we start by studying the bifurcation branches for $F^{N_1}$.
The subspace $\{(x_1,x_2,x_3)\in\RR^3: x_2=x_3\}$ is invariant by every $N_1$-admissible system $F^{N_1}$ and the bifurcation condition $g_x+h_1=0$ still holds in this subspace. 
Thus, we look for bifurcation branches in this subspace where a system $F^{N_1}$ has the following form:
$$\begin{cases}
\dot{x}_1=g(x_1,\lambda)\\
\dot{x}_2=g(x_2,\lambda)+ h(x_2,x_1,\lambda)\\  
\end{cases}.$$

Assuming that $g_x\neq 0$, we conclude that the first cell has only the equilibrium $0$ near the origin.
Then, we look for the solutions of the equation 
\begin{equation}\label{eq:ex_gener_4}
g(x_2,\lambda)+ h(x_2,0,\lambda)=0.
\end{equation}

If the second derivatives satisfy the non-degeneracy conditions 
$g_{xx}+h_{11}\neq 0$ and $g_{x\lambda}+h_{1\lambda}\neq 0$
, then there exists $b:]-\epsilon,\epsilon[\rightarrow \RR$ such that 
$$g(x_2,\lambda)+ h(x_2,0,\lambda)=0\iff x_2=0 \vee x_2= b(\lambda),$$
in a neighbourhood of the origin. 
Moreover, we know that 
$$b'(0)= -2\dfrac{g_{x\lambda}+h_{1\lambda}}{g_{xx}+h_{11}}.$$

Now, we see how the bifurcation branch for $N_1$ can be extended to $N$.
Note that $\underline{\overline{J}}_F^{N_2}$ has rank $1$ and  
$$-2\dfrac{g_{x\lambda}+h_{1\lambda}}{g_{xx}+h_{11}}\begin{bmatrix}
-h_1\\
-h_1
\end{bmatrix}\notin Im(\begin{bmatrix}
2h_1&-2h_1\\
-h_1 & h_1
\end{bmatrix}).$$
Thus there exists $z=(z_5,z_{\lambda}):]-\epsilon,\epsilon[\rightarrow \RR\times \RR$ such that 
$$F^{N_2}(b(\lambda),x_4,x_5,\lambda)=0\iff (x_5,\lambda)= z(x_4).$$
Since $(0,1,1)$ is an eigenvalue of $J_F^{N_2}$ associated with the bifurcation condition, we have that $z'_5 =1$ and we know that $z'_\lambda =0$. 

The Hessian matrix of $f_{4}^{N_2}$ and $f_{5}^{N_2}$ are given by
$$H_{f_{4}}^{N_2}=\begin{bmatrix}
h_{22}&h_{12}&0\\
h_{12}& g_{xx}+3h_{11}& 2h_{12} \\
 0 		& 2h_{12}& 2h_{22} 
\end{bmatrix}\quad \quad H_{f_{5}}^{N_2}=\begin{bmatrix}
h_{22}&		0		& h_{12}\\
0			& h_{22}& h_{12} \\
h_{12}& h_{12}& g_{xx}+2h_{11} 
\end{bmatrix}$$
and the matrix $H$ in Proposition~\ref{prop:extbifbra} is given by
$$H= \left[\begin{bmatrix}1 & 1\end{bmatrix} \underline{\overline{H}}^{N_2}_{f_j} \begin{bmatrix}1 \\ 1\end{bmatrix} \right]_{j=4,5}= \begin{bmatrix}g_{xx}+3h_{11}+4h_{12}+2h_{22} \\ g_{xx}+2h_{11}+2h_{12}+h_{22}\end{bmatrix}.$$
Since $h(x,x,\lambda)=0$, we have that $h_{11}+2h_{12}+h_{22}=0$ and
$$H=\begin{bmatrix}g_{xx}+h_{11} \\ g_{xx}+h_{11}\end{bmatrix}\notin Im(\underline{\overline{J}}_F^{N_2}).$$
Now, it follows from Proposition~\ref{prop:extbifbra} that the bifurcation branch $(0,b(\lambda),b(\lambda))$ can be extended to two bifurcation branches in the coalescence network $N$. Moreover, these bifurcation branches grow as $\sqrt{|\lambda|}$ in the cells $4$ and $5$. 

We calculate the second derivatives of the function $z$ mentioned above, we have that 
$$-2 z''_{\lambda}(0) \dfrac{g_{x\lambda}+h_{1\lambda}}{g_{xx}+h_{11}}\begin{bmatrix}h_1 \\ h_1\end{bmatrix}+ \begin{bmatrix}
2h_{1} & -2h_1\\
-h_{1} & h_{1}  
\end{bmatrix}\begin{bmatrix}0\\ z''_5(0)\end{bmatrix} 
+ \begin{bmatrix}g_{xx}+h_{11} \\ g_{xx}+h_{11}\end{bmatrix} =0$$
%$$\iff
 %\begin{cases}
%-2 z''_{\lambda}(0) \dfrac{g_{x\lambda}+h_{1\lambda}}{g_{xx}+h_{11}}h_1-2h_x z''_5(0)+g_{xx}+h_{11}=0\\ 
%-2z''_{\lambda}(0) \dfrac{g_{x\lambda}+h_{1\lambda}}{g_{xx}+h_{11}}h_1+h_{x}z''_5(0)+g_{xx}+h_{11}=0
%\end{cases}$$
%$$\iff
 %\begin{cases}
%-2z''_{\lambda}(0) \dfrac{g_{x\lambda}+h_{1\lambda}}{g_{xx}+h_{11}}h_1+g_{xx}+h_{11}=0\\ 
%3h_{x}z''_5(0)=0
%\end{cases}$$
$$\iff
 \begin{cases}
z''_{\lambda}(0) = \dfrac{(g_{xx}+h_{11})^2}{2h_1(g_{x\lambda}+h_{1\lambda})}\\ 
z''_5(0)=0
\end{cases}.$$
Thus, generically, the second derivative of $z_{\lambda}$ does not vanish.
Moreover, the bifurcation branches in the coalescence network $N$ are defined for positive or negative values of $\lambda$, depending if $h_1(g_{x\lambda}+h_{1\lambda})$ is positive or negative.
\end{exam}

Now, we look at an example that satisfies the first two conditions in Proposition~\ref{prop:extbifbra}, that is, $\underline{\overline{J}}_F^{N_2}$ has rank $n_2-2$ and $(b_c)'(0)L_{N_2}^c\notin Im(\underline{\overline{J}}_F^{N_2})$, but does not satisfy the third condition on the vector $H$.
Despite this, we still have that each bifurcation branch of $F^{N_1}$ extends to two bifurcation branches of $F^{N}$. However, the conclusions on the growth-rates do not apply.

\begin{exam}\label{ex_gener_1_bif}
Let $N=N_1\circ N_2$ be the feedforward coalescence network given in Example~\ref{exam:ex_gener_1} and let $F^N$ be an $N$-admissible diffusive system such that $g_x+2h_1=0$ and $F^{N}(0,0,0,0,0,\lambda)=0$. 
First, we study the bifurcation branches for $F^{N_1}$. 
Assuming that $g_x\neq 0$ and $g_x+h_1\neq 0$, we conclude that the first two cells have only the equilibrium $0$ near the origin.
If the second derivatives satisfy the non-degeneracy conditions $g_{xx}+2h_{11}\neq 0$ and $g_{x\lambda}+2h_{1\lambda}\neq 0$, then there exists $b:]-\epsilon,\epsilon[\rightarrow \RR$ such that 
$$g(x_3,\lambda)+ 2h(x_3,0,\lambda)=0\iff x_3=0 \vee x_3= b(\lambda),$$
in a neighbourhood of the origin. 
Moreover, we know that 
$$b'(0)= -2\dfrac{g_{x\lambda}+2h_{1\lambda}}{g_{xx}+2h_{11}}.$$

Thus there is the bifurcation branch for $N_1$ given by $(0,0,b(\lambda))$, which we extend to the coalescence network $N$. 

Note that $\underline{\overline{J}}_F^{N_2}$ has rank $1$ and  
$$-2\dfrac{g_{x\lambda}+2h_{1\lambda}}{g_{xx}+2h_{11}}\begin{bmatrix}
-2h_x\\
-h_x
\end{bmatrix}\notin Im(\begin{bmatrix}
0&0\\
-h_x & 0
\end{bmatrix}).$$
Thus there exists $z=(z_4,z_{\lambda}):]-\epsilon,\epsilon[\rightarrow \RR\times \RR$ such that 
$$F^{N_2}(b(\lambda),x_4,x_5,\lambda)=0\iff (x_4,\lambda)= z(x_5).$$
Since $(0,0,1)$ is an eigenvector of $J_F^{N_2}$ associated with the bifurcation condition, we have that $z'_4 =0$ and we know that $z'_\lambda =0$. 

The Hessian matrix of $f_{4}^{N_2}$ and $f_{5}^{N_2}$ are given by
$$H_{f_{4}}^{N_2}=\begin{bmatrix}
2h_{22}&2h_{12}&0\\
2h_{12}& g_{xx}+2h_{11}& 0 \\
 0 		& 0 & 0 
\end{bmatrix}\quad \quad H_{f_{5}}^{N_2}=\begin{bmatrix}
h_{22}&		0		& h_{12}\\
0			& h_{22}& h_{12} \\
h_{12}& h_{12}& g_{xx}+2h_{11} 
\end{bmatrix}$$
and the matrix $H$ in Proposition~\ref{prop:extbifbra} is given by
$$H= \left[\begin{bmatrix}0 & 1\end{bmatrix} \underline{\overline{H}}^{N_2}_{f_j} \begin{bmatrix}0 \\ 1\end{bmatrix} \right]_{j=4,5}= \begin{bmatrix}0 \\ g_{xx}+2h_{11}\end{bmatrix}\in Im(\underline{\overline{J}}_F^{N_2}).$$
It follows from the calculation made before Proposition~\ref{prop:extbifbra} that $z''_\lambda =0$ and that $z''_4=-(g_{xx}+2h_{11})/h_1$.

Taking the third and fourth derivatives of $F^{N_2}(b(z_\lambda(x_5)),z_4(x_5),x_5,z_\lambda(x_5))$ at the origin which must be zero, we conclude that
$$z'''_\lambda =0, \quad z^{(4)}_\lambda= \dfrac{3(g_{xx}+2h_{11})^4}{4h_1^{3}(g_{x\lambda}+2h_{1\lambda})}\neq 0.$$

Since the fourth derivative of $z_\lambda$ does not vanishes, generically, we conclude that the bifurcation branch $(0,0,b(\lambda))$ can be extended to two bifurcation branches in the coalescence network $N$. Moreover, these bifurcation branches grow as $\sqrt[4]{|\lambda|}$ in the fifth cell and they grow as $\sqrt[2]{|\lambda|}$ in the fourth cell, since $z'_4=0$ and $z''_4\neq 0$. 
Furthermore, $z^{(4)}_\lambda$ determinates if those bifurcation branches in the coalescence network $N$ are defined for positive or negative values of $\lambda$.

There is also the trivial equilibrium branch in $N_1$ which extends to the trivial equilibrium branch in $N$ and the bifurcation branch $(0,0,0,0,0,b(\lambda))$.
\end{exam}

Next, we consider that $\underline{\overline{J}}_F^{N_2}$ has rank $n_2-2$ but $L_{N_2}^c\in Im(\underline{\overline{J}}_F^{N_2})$. 
In this case, we know that $rank(D_{x\lambda}\phi)=rank(D_{x}\phi)$ and we need to study a different bifurcation problem.
This new problem concerns replacing the input from cell $c$ with the bifurcation branch $b_c(\lambda)$ in the network $N_2$ .

Consider the function $\phi:\RR^{n_2-1}\times \RR\rightarrow\RR^{n_2-1}$ given by $\phi(x_2,\dots,x_{n_2},\lambda)= \overline{F}^{N_2}(b_c(\lambda),x_2,\dots,x_{n_2},\lambda)$.
Applying the Lyapunov-Schmidt Reduction, we know that there exists a function $W:\RR\times\RR\rightarrow \Im(\underline{\overline{J}}_F^{N_2})$  such that 
$$\phi(x,\lambda)=0 \iff x=yv+W(y ,\lambda) \wedge \langle v^*,\phi(yv+W(y ,\lambda),\lambda)\rangle =0,$$  
where $v\in Ker(\underline{\overline{J}}_F^{N_2})$, $v^*\in \Im(\underline{\overline{J}}_F^{N_2})^{\bot}$ and $y,\lambda\in\RR$ such that $\|v\|=\|v^*\|=1$. 
%The derivatives of $W$ are given by
%$$W_\lambda=-(\overline{J}_F^{N_2})^{-1} E \phi_\lambda,\quad \quad W_y=0, \quad  \quad  W_{\lambda\lambda}= - (\overline{J}_F^{N_2})^{-1} E(D^2\phi(W_{\lambda},W_{\lambda})+2D_{x}\phi_{\lambda} W_{\lambda}), $$
%$$W_{y\lambda}= - (\overline{J}_F^{N_2})^{-1} E (D^2\phi(v,W_{\lambda})+D_{x}\phi_{\lambda} v) ,\quad W_{yy}=- (\overline{J}_F^{N_2})^{-1} E D^2\phi(v,v), $$
%where $E$ is the projection into $\Im(\underline{\overline{J}}_F^{N_2})$ with kernel equal to $\Im(\underline{\overline{J}}_F^{N_2})^{\bot}$
%\footnote{Se $v^*\in Ker(\underline{\overline{J}}_F^{N_2})$, então aqui o kernel tem de ser igual a $Ker(\underline{\overline{J}}_F^{N_2})$.}.

The first derivatives of $\psi(y,\lambda)= \langle v^*,\phi(yv+W(y ,\lambda),\lambda)\rangle$ at the origin vanish 
$$\psi_y=0,\quad \quad \psi_{\lambda}=\langle v^*, (b_c)'(0)[J_F^{N_2}]^c\rangle=0.$$

%If $(b_c)'(0)\neq 0$, then the derivatives of $\psi$ with respect to $\lambda$ multiple times are, generically, zero. 
For any $k\in\mathbb{N}$, the derivative of $\phi$ with respect to $\lambda$, $k$ times, is equal to 
\begin{equation}\label{eq:derphilam}
\phi_{\lambda^k}= \sum_{j=0}^{k}\sum_{i=0}^{k-j}\sum_{\substack{
l_1,\dots,l_i\geq 0\\
j+l_1+\dots+l_i=k
}} \alpha_{j,i,l_1,\dots,l_i} (b_c)^{(l_1)}(0)\dots(b_c)^{(l_i)}(0) \overline{F}^{N_2}_{\lambda^j x_c^{i}},
\end{equation}
 for some constants $\alpha_{j,i,l_1,\dots,l_i}\in\mathbb{N}$, $(b_c)^{(l)}(0)\in\RR$ the $l$-derivative of $b_c$ at zero and $\overline{F}^{N_2}_{\lambda^j x_c^{i}}$ the derivative of $\overline{F}^{N_2}$ with respect to $x_c$ $i$-times and to $\lambda$ $j$-times.
Note that $\overline{F}^{N_2}_{\lambda^j}=0$ and 
$$\overline{F}^{N_2}_{\lambda^j x_c^{i}}= h_{x^{i}\lambda^{j}}L_{N_2}^c,$$
where $h_{x^{i}\lambda^{j}}$ is the derivative of $h(x,y,\lambda)$ with respect to $x$ $i$-times and to $\lambda$ $j$-times.
Since $L_{N_2}^c\in Im(\underline{\overline{J}}_F^{N_2})$ and $v^*\in \Im(\underline{\overline{J}}_F^{N_2})^{\bot}$, we have that
$$\psi_{\lambda^k}=\langle v^*, \phi_{\lambda^k} \rangle=0.$$ 

The second derivatives of $\phi$ at the origin are
$$\psi_{yy}=\langle v^*, H \rangle,$$ 
with $H$ the vector defined above, and given by $D^2\phi(v,v)=\left[v^{T} \underline{\overline{H}}^{N_2}_{f_j} v \right]_{j\in C_2\setminus\{c\}}$, and
$$\psi_{y\lambda}=\langle v^*,D^2\phi(v,-(D_x\phi)^{-1}(b_c)'(0)[J_F^{N_2}]^c)\rangle + \langle v^*, D_{x}\phi_{\lambda} v \rangle,$$
%$$=\langle v^*,D^2\phi(v,W_\lambda) + D_x\phi W_{x\lambda} + D_{x}\phi_{\lambda} v \rangle$$
%$$=\langle v^*,D^2\phi(v,-(D_x\phi)^{-1}E(b_c)'(0)[J_F^{N_2}]^c)-D_x\phi (D_x\phi)^{-1}E D^2(v,-(D_x\phi)^{-1}E(b_c)'(0)[J_F^{N_2}]^c)) + D_{x}\phi_{\lambda} v \rangle= ,$$
where $D^2\phi(v,w)=\left[v^{T} \underline{\overline{H}}^{N_2}_{f_j} w \right]_{j\in C_2\setminus\{c\}}$ and $D_{x}\phi_{\lambda} v$ is the derivative with respect one time to variables $x_2,\dots, x_{n_2}$ and one time to $\lambda$.

Note that $D_{x}\phi_{\lambda}= \underline{\overline{J}}_{f_{\lambda}}^{N_2}= g_{x\lambda}Id+ h_{x\lambda}\underline{\overline{L}}_{N_2}$ and
$$\psi_{y\lambda}=\langle v^*,D^2\phi(v,-(D_x\phi)^{-1}(b_c)'(0)[J_F^{N_2}]^c)\rangle + (g_{x\lambda}+\mu h_{x\lambda})\langle v^*, v \rangle.$$
Moreover, $\langle v^*, v \rangle\neq 0$ if and only if $ v \notin \Im(\underline{\overline{J}}_F^{N_2})$ if and only if the eigenvalue $\mu$ is simple in $\underline{\overline{L}}_{N_2}$.
In this case, generically, $\psi_{y\lambda}\neq 0$, because the derivative $g_{x\lambda}$ only appears in the second inner product of $\psi_{y\lambda}$.

We can, thus, state the following result.

\begin{prop}\label{prop:extbif2}
Let $N=N_1\circ N_2$ be a feedforward coalescence network, $F^N$ an $N$-admissible diffusive system and $\mu$ an eigenvalue of both $L_{N_1}$ and $\underline{\overline{L}}_{N_2}$.
Assume that $\mu$ is simple for $\underline{\overline{L}}_{N_2}$, that $(b_c)'(0)L_{N_2}^c\in Im(\underline{\overline{J}}_F^{N_2})$ and that $H\notin Im(\underline{\overline{J}}_F^{N_2})$.
If  $(b_1,b_c):D_1\rightarrow \RR^{n_1-1}\times \RR$ is a bifurcation branch for $F^{N_1}$ then, generically, there exist $b^1_2, b^2_2: D_1\rightarrow \RR^{n_2-1}$ such that $b^1=(b_1,b_c,b^1_2), b^2=(b_1,b_c,b^2_2):D_1\rightarrow \RR^{n_1+n_2-1}$ are bifurcation branches for $F^{N}$. 

Moreover, supposing that the bifurcation branch $(b_1,b_c)$ has linear growth and $(b_c)'(0)\neq 0$, then, generically, the extended bifurcation branches also have linear growth.
\end{prop}

\begin{proof}
It follows from the discussion above that we need to study the equation $\psi(y,\lambda)=0$. We also know that, generically, 
$$\psi_y=0,\quad \psi_{\lambda^k}=0, \quad \psi_{yy}\neq 0 \textrm{ and } \psi_{y\lambda}\neq 0.$$

Taylor expanding the equation, we have
$$\frac{\psi_{yy}}{2}y^2+\psi_{y\lambda} y\lambda + y\mathcal{O}(2)=0\iff y=0 \vee \frac{\psi_{yy}}{2}y+\psi_{y\lambda} \lambda + \mathcal{O}(2)=0$$
Since $\psi_{yy}\neq 0$, the implicit function theorem implies that there exists $y:D_1\rightarrow \RR$ such that 
$$\frac{\psi_{yy}}{2}y(\lambda)+\psi_{y\lambda} \lambda + \mathcal{O}(2)=0.$$

Taking $b^1_2(\lambda)= W(0,\lambda)$ and $b^2_2(\lambda)= y(\lambda)v+ W(y(\lambda),\lambda)$, we obtain the following bifurcation branches of $F^{N}$:  $b^1=(b_1,b_c,b^1_2)$ and $b^2=(b_1,b_c,b^2_2)$. 
Differentiating the last equation, we see that, generically, $y'(0)\neq 0$ and the extension of bifurcation branch $b^2_2$ has a linear growth.
The derivative of $W(0,\lambda)$ is equal to $-b'_c(0)h_1(\underline{\overline{J}}_F^{N_2})^{-1}L_{N_2}^c$ which does not vanishes if $(b_c)'(0)\neq 0$.
Therefore, if the bifurcation branch $(b_1,b_c)$ for $N_1$ has linear growth, then the extended bifurcation branches for $N$ preserve the linear growth. 
\end{proof}

In the next example, we apply the result presented in Proposition~\ref{prop:extbif2}.

\begin{exam}
Consider the feedforward coalescence network $N= N_1\circ N_2$ given in Figure~\ref{Fig:ex_gener_5} and an $N$-admissible diffusive system $F^N$ such that $g_x+3h_1=0$ and $F^{N}(0,0,0,0,0,\lambda)=0$.
Generically, a transcritical bifurcation occurs in the network $N_1$. Thus, we have a bifurcation branch $b=(b_1,b_2)$ for $F^{N_1}$. 
We consider this branch and use the previous result to extend it to two bifurcation branches in $N$.

We need to check that $3$ is a simple eigenvalue of $\underline{\overline{L}}_{N_2}$, that $L_{N_2}^c\in Im(\underline{\overline{J}}_F^{N_2})$ and that $H\notin Im(\underline{\overline{J}}_F^{N_2})$. 
We have that  
$$\underline{\overline{L}}_{N_2}=\left[
\begin{matrix}
8&-10\\
5& -7
\end{matrix}\right] $$
has the eigenvalues $3$ and $-2$ with associated eigenvectors $(2, 1)$ and $(1, 1)$, respectively. So  $\underline{\overline{J}}_F^{N_2}$ has the eigenvalues $0$ and $-5h_1$ and the image of $\underline{\overline{J}}_F^{N_2}$ is generated by $(1, 1)$. Thus $3$ is a simple eigenvalue of $\underline{\overline{L}}_{N_2}$ and $L_{N_2}^c\in Im(\underline{\overline{J}}_F^{N_2})$. The vector $H$ is given by
$$H=\left[
\begin{matrix}
\frac{2}{3} \left(2 g_{xx}+10 h_{12}+11 h_{11}\right)\\
\frac{1}{3} \left(g_{xx}+20 h_{12}+13 h_{11}\right)
\end{matrix}\right] $$
which does not belong to $Im(\underline{\overline{J}}_F^{N_2})$.
Therefore, it follows from Proposition~\ref{prop:extbif2} that the bifurcation branch $b$ for $N_1$ extends to two bifurcation branches in $N$.
\end{exam}

\begin{figure}[!h]
\begin{center}
\begin{tabular}{ccc}
\begin{tikzpicture}
 [scale=.25,auto=left, node distance=1.5cm]
 \node[style={circle,draw}] (n2) at (0,0) {\small{$\overline{2}$}};
  \node[style={circle,draw}] (n1) at (0,8) {\small{1}};
  \path
  (n1) edge[->,bend left=10, thick]  node  {} (n2)
 (n2) edge[->,bend left=10, thick]  node  {{\footnotesize $2$}} (n1); 
\end{tikzpicture}  \qquad & \qquad
\begin{tikzpicture}
 [scale=.25,auto=left, node distance=1.5cm]
 \node[style={circle,draw}] (n3) at (6,8) {\small{3}};
  \node[style={circle,draw}] (n2) at (0,4) {\small{$\underline{2}$}};
 \node[style={circle,draw}] (n4) at (6,0) {\small{4}};
  \path
	(n2) edge[->, thick]  node  {{\footnotesize $-2$}} (n3)
	(n2) edge[->, thick]  node   [below=1pt] {{\footnotesize $-2$}} (n4)
  (n3) edge[->, bend left=10, thick]  node   {{\footnotesize $-5$}} (n4)
	(n4) edge[->, bend left=10, thick]  node   {{\footnotesize $10$}} (n3); 
\end{tikzpicture}  \qquad & \qquad
\begin{tikzpicture}
 [scale=.25,auto=left, node distance=1.5cm]
\node[style={circle,draw}] (n1) at (0,4) {\small{1}};
\node[style={circle,draw}] (n2) at (8,4) {\small{2}};
\node[style={circle,draw}] (n3) at (14,8) {\small{3}};
\node[style={circle,draw}] (n4) at (14,0) {\small{4}};
  \path
  (n1) edge[->,bend left=10, thick]  node  {} (n2)
 (n2) edge[->,bend left=10, thick]  node  {{\footnotesize $2$}} (n1) 
(n2) edge[->, thick]  node  {{\footnotesize $-2$}} (n3)
	(n2) edge[->, thick]  node   [below=1pt] {{\footnotesize $-2$}} (n4)
  (n3) edge[->, bend left=10, thick]  node   {{\footnotesize $-5$}} (n4)
	(n4) edge[->, bend left=10, thick]  node   {{\footnotesize $10$}} (n3); 
 \end{tikzpicture} 
\end{tabular}
\end{center}
\caption{The network $N_1 \circ N_2$  (right) is the coalescence of the two network $N_1$ and $N_2$ (left and center), by merging cells $\overline{2}$ and $\underline{2}$.} 
\label{Fig:ex_gener_5}
\end{figure}
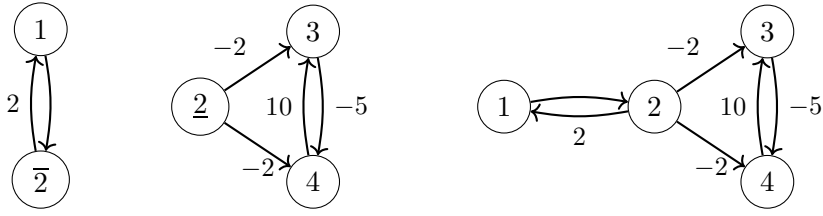

Note that the first part of the result in Proposition~\ref{prop:extbif2} still holds when  $L_{N_2}^c\notin Im(\underline{\overline{J}}_F^{N_2})$ provided that $b'_c(\lambda)=0$, as we exemplify next.

\begin{exam}
Consider the feedforward coalescence network $N= N_1\circ N_2$ given in Figure~\ref{Fig:ex_gener_4}, an $N$-admissible diffusive system $F^N$ such that $g_x+h_1=0$ and $F^{N}(0,0,0,0,0,\lambda)=0$ and the trivial branch $b(\lambda)=0$ in $N_1$.
We saw in Example~\ref{exam:ex_gener_4b} that this branch extends to two branches in $N$: the trivial branch and a bifurcation branch.
Taking the trivial branch, it means that all derivatives of $b_c$ are zero. 
It follows from (\ref{eq:derphilam}) and $\overline{F}^{N_2}_{\lambda^j}=0$ that $\psi_{\lambda^k}=0$.
Therefore, if the eigenvalue $\mu$ is simple in $\underline{\overline{L}}_{N_2}$ and $H\notin Im(\underline{\overline{J}}_F^{N_2})$, we can apply Proposition~\ref{prop:extbif2} to also conclude that the trivial branch in $N_1$ extends to two branches in $N$ (one of them trivial).

The eigenvalue $1$ is simple in $\underline{\overline{L}}_{N_2}$, see Example~\ref{exam:ex_gener_4}. 
The matrix $H$ is equal to the matrix $H$ calculated in Example~\ref{exam:ex_gener_4b}, because the vector $v$ is the same in both cases.
The trivial branch in $N_1$ extends to two branches in $N$, since 
\[H=\begin{bmatrix}g_{xx}+h_{11} \\ g_{xx}+h_{11}\end{bmatrix}\notin Im(\underline{\overline{J}}_F^{N_2}).\]

Remembering the Example~\ref{exam:ex_gener_4b}, we conclude that any generic $N$-admissible diffusive system $F^N$ such that $g_x+h_1=0$ and $F^{N}(0,0,0,0,0,\lambda)=0$ has three bifurcation branches where two of them have a square-root growth rate and one has a linear growth rate.
\end{exam}

\begin{rem}
In the proof of proposition~\ref{prop:extbif2}, if we remove the condition that $H\notin Im(\underline{\overline{J}}_F^{N_2})$, we still have that
$$\psi_y=0,\quad \psi_{\lambda^k}=0, \textrm{ and } \psi_{y\lambda}\neq 0.$$   
Thus the equation $\frac{\psi_{yy}}{2}y+\psi_{y\lambda} \lambda + \mathcal{O}(2)=0$ can still be solved in order to $\lambda$, i.e. there exists $\Lambda(y)$ such that  $\psi(y,\Lambda(y))=0$, by the implicit function theorem.
It does not provide a straightforward bifurcation branch, we still need to reparametrize the curve to understand the number of bifurcation branches and their growth rate.
This reparemetrization can be done for given networks, as we do next, however the assumptions needed for a general result do not enlighten us about the network structure.
\end{rem}

We finish by looking at the network in Figure~\ref{Fig:coal_3-cell_net6} which is the network $6$ in ~\cite{ADGL09}.
In~\cite{ADGL09}, the authors studied this network as a lift of the $3$-cell bidirectional ring network and showed that it supports more bifurcation branches that the ones lifted from the $3$-cell network. 
They describe the bifurcation branches in \cite[Theorem 3.4]{ADGL09}.
Here, we look at the network $6$ as a FFCN and apply our results to study the bifurcation branches.
In this case $H\in Im(\underline{\overline{J}}_F^{N_2})$ and we need to study higher order derivatives.
 
\begin{exam} \label{ex:Fig_6}
Consider the feedforward coalescence network $N= N_1\circ N_2$ given in Figure~\ref{Fig:coal_3-cell_net6}.
The eigenvalue $3$ is shared by the networks $N_1$ and $N_2$, and it is simple in $\underline{\overline{L}}_{N_2}$ and $L_{N_2}^c\in Im((\underline{\overline{L}}_{N_2} - 3I))$, see Example~\ref{exam:net6}.
Let $F^N$ be an $N$-admissible diffusive system with the bifurcation condition $g_x+3h_x=0$ and assume without loss of generality that $F^{N}(0,0,0,0,0,\lambda)=0$. 
The network $N_1$ has $S_3$ symmetry and there are $3$ bifurcation branches for $F^{N_1}$ associated with the bifurcation condition $g_x+3h_x=0$.
The network $N_1$ also supports the trivial equilibrium branch. 
Taking one of these branches and considering the function $\psi(y,\lambda)= \langle v^*,\phi(yv+W(y ,\lambda),\lambda)\rangle$ as defined before, we have that
$$\psi_y=0,\quad \psi_{\lambda^k}=0, \textrm{ and } \psi_{y\lambda}\neq 0,$$
for any $k\in\mathbb{N}$.

We have that  $D^2\phi(v,v)=\frac{g_{xx}-4 h_{12}+h_{11}}{2} (1,1)\in Im(\underline{\overline{J}}_F^{N_2})$, so $\psi_{yy}=0$. 
Moreover, we can calculate the third derivative 
 $$\begin{aligned}
\psi_{yyy}&=\langle v^*, D^3 \phi(v,v,v)+3 D^2\phi (v,W_{yy})+ D\phi(W_{yyy})\rangle\\
&= \frac{6 g_{xx} \left(h_{12}-2 h_{11}\right)+2 g_{xxx} h_1-3 g_{xx}^2+24 h_{12}^2-9 h_{11}^2+12 h_{1} h_{122}+30 h_{12} h_{11}+6 h_{1} h_{111}}{4 h_{1}}\\
&\neq 0,
\end{aligned}$$
where $W_{yy}=-(\underline{\overline{J}}_F^{N_2})^{-1} H$.

Therefore, 
$$\begin{aligned}
\psi(y,\lambda)=0 &\Leftrightarrow \psi_{y\lambda}y\lambda+\frac{\psi_{yyy}}{6}y^3+\frac{\psi_{yy\lambda}}{2}y^2\lambda+\frac{\psi_{y\lambda\lambda}}{2}y\lambda^2 +y\mathcal{O}(3)=0\\
%									&\Leftrightarrow y(\psi_{y\lambda}\lambda+\frac{\psi_{yyy}}{6}y^2+\frac{\psi_{yy\lambda}}{2}y\lambda+\frac{\psi_{y\lambda\lambda}}{2}\lambda^2 +\mathcal{O}(3))=0\\
									&\Leftrightarrow y=0\vee (\psi_{y\lambda}\lambda+\frac{\psi_{yyy}}{6}y^2+\frac{\psi_{yy\lambda}}{2}y\lambda+\frac{\psi_{y\lambda\lambda}}{2}\lambda^2 +\mathcal{O}(3))=0.
\end{aligned}$$

Define the function $\tilde{\psi}(y,\lambda)= \psi_{y\lambda}\lambda+\frac{\psi_{yyy}}{6}y^2+\frac{\psi_{yy\lambda}}{2}y\lambda+\frac{\psi_{y\lambda\lambda}}{2}\lambda^2 +\mathcal{O}(3)$. Since $\tilde{\psi}_{\lambda}= \psi_{y\lambda}\neq 0$, we know by the implicit function theorem that there exists a function $\Lambda(y)$ such that  $\tilde{\psi}(y,\Lambda(y))=0$. Differentiating twice the previous equality, we obtain 
$$\Lambda'(0)=0, \quad \quad \Lambda''(0)=-\dfrac{\tilde{\psi}_{yy}}{\tilde{\psi}_{\lambda}}=-\dfrac{\psi_{yyy}}{3\psi_{y\lambda}}\neq 0.$$
Therefore, we can rewrite the solution $(y,\Lambda(y))$ as two branches $(y_1(\lambda),\lambda)$ and $(y_2(\lambda),\lambda)$, for positive or negative values of $\lambda$, depending on the sign of $\Lambda''(0)$. 
Thus for each equilibrium branch of the network $N_1$, we obtain three equilibrium branches on the coalescence network. 

The $3$ bifurcation branches for $F^{N_1}$ extend to $9$ bifurcation branches in $N$. 
Moreover, the trivial equilibrium branch in $N_1$ extends to $2$ bifurcation branches in $N$ and a trivial equilibrium branch in $N$. 
Therefore there are $11$ bifurcation branches for $F^{N}$ and the trivial equilibrium branch.
In \cite[Theorem 3.4]{ADGL09}, the result states that there are $9$ bifurcation branches, however the difference corresponds to a different way of parametrizing and counting the branches.
Here, we consider that a saddle node bifurcation creates two branches and in \cite{ADGL09} the authors count this as one branch.
Moreover, the bifurcation branches have a square-root growth rate. 
\end{exam}

\section{Final remarks}

In this work we consider networks that can be given as the coalescence of subnetworks and analyse the coalescence bifurcation problem -- given a steady-state bifurcation for the coalescence network $N_1 \circ N_2$, can we obtain the information about the bifurcating branches by analysing the bifurcation in one or the two component networks, $N_1$, $N_2$? 
After concluding that this is not possible for general coalescence networks, we restrict to Feedforward Coalescence Networks (FFCNs) and consider admissible diffusive systems.
We show that the answer is positive when the bifurcation occurs in only one of the component networks.
When the bifurcation occurs for both component networks, we give conditions for the bifurcation branches in $N_1$ to extend to bifurcation branches in $N$. Those conditions rely on the coupling structure of $N_2$.
We have shown in Example~\ref{ex:Fig_6} that the methods introduced here apply to the network 6 in \cite{ADGL09}, which is a FFCN, and that we can obtain all the bifurcating branches for the network 6 from those of its component networks.

Our results on the coalescence bifurcation problem consider FFCNs with two component networks but they generalize for FFCNs given by the sequential coalescence of more than two networks.
However, we need to carefully analyse the domain where the bifurcation branches are defined because some extensions can dictate that the bifurcation branch is defined only for positive parameters and a subsequent extension can force the bifurcation branch to be defined only for negative parameters, and vice-versa.
This problem can be addressed by defining an axiomatic set which characterizes the bifurcation branches like it is done in \cite[Proposition 5.3]{S18}.

%$$\begin{aligned}
%\psi_{yy\lambda}=&\langle v^*, D^3 \phi(v,v,W_\lambda)+ D^2 \phi_{\lambda}(v,v)+2 D^2\phi (v,W_{y\lambda})+ D^2\phi(W_{yy},W_\lambda)+ D\phi_{\lambda}.W_{yy}+ D\phi.W_{yy\lambda}\rangle\\
								%=&\langle v^*, D^3 \phi(v,v,-(D_x\phi)^{-1}(b_c)'(0)[J_F^{N_2}]^c)+ D^2 \phi_{\lambda}(v,v)+ 2D^2\phi (v,-(D_x\phi)^{-1}(D^2\phi(v,W_{\lambda})+D\phi_{\lambda}v))+\\
								 %&\quad  D^2\phi(-(D_x\phi)^{-1}D^2\phi(v,v),-(D_x\phi)^{-1}(b_c)'(0)[J_F^{N_2}]^c)+ D\phi_{\lambda}(-(D_x\phi)^{-1}D^2\phi(v,v))\rangle
%\end{aligned}$$
%
%$$\begin{aligned}
%\psi_{y\lambda\lambda}=\langle v^*,& D^3 \phi(v,W_\lambda,W_\lambda)+2 D^2 \phi_{\lambda}(v,W_{\lambda})+ 2 D^2\phi (W_{y\lambda},W_{\lambda})+ D^2\phi (W_{y},W_{\lambda\lambda})\\
											%&+D\phi_{\lambda\lambda}(v)+2D\phi_{\lambda}(W_{y\lambda})+D\phi(W_{y\lambda\lambda})\rangle\\
											%= \langle v^*,& D^3 \phi(v,W_\lambda,W_\lambda)+2 D^2\phi (W_{y\lambda},W_{\lambda})+ 2 D^2 \phi_{\lambda}(v,W_{\lambda})+2D\phi_{\lambda}(W_{y\lambda})+D\phi_{\lambda\lambda}(v)\rangle\\
						%=\langle v^*, &   D^3 \phi(v,-(D_x\phi)^{-1}(b_c)'(0)[J_F^{N_2}]^c,-(D_x\phi)^{-1}(b_c)'(0)[J_F^{N_2}]^c)\\													
													%&+ 2D^2\phi (-(D_x\phi)^{-1}(D^2\phi(v,-(D_x\phi)^{-1}(b_c)'(0)[J_F^{N_2}]^c)+D\phi_{\lambda}v),-(D_x\phi)^{-1}(b_c)'(0)[J_F^{N_2}]^c)\\
													%&+ 2D^2 \phi_{\lambda}(v,-(D_x\phi)^{-1}(b_c)'(0)[J_F^{N_2}]^c)\\
													%&+2D\phi_{\lambda}(-(D_x\phi)^{-1}(D^2\phi(v,-(D_x\phi)^{-1}(b_c)'(0)[J_F^{N_2}]^c)+D\phi_{\lambda}v))\\
													%&+D\phi_{\lambda\lambda}(v)\rangle														
											%\end{aligned}$$

%%%%%%%%%%%%%%%%%%%%%%%%%%%%%%%%%%
%%%%%%%%%%%%%%%%%%%%%%%%%%%%%%%%%%

\vspace{5mm}

\noindent {\bf Acknowledgments} \\
MA was partially supported by CMUP, member of LASI, which is financed by national funds through FCT -- Funda\c c\~ao para a Ci\^encia e a Tecnologia, I.P., under the projects with reference UIDB/00144/2020 and UIDP/00144/2020.

\end{document}